\documentclass[11pt]{article}
\usepackage{amsthm}
\usepackage{amscd}
\usepackage{mathpazo}
\usepackage{txfonts}

\usepackage{amsbsy,amsmath}
\usepackage{mathrsfs}
\usepackage{enumerate}

\usepackage{mathtools}
\allowdisplaybreaks[1]
\usepackage{color}
\usepackage{hyperref}
\usepackage{lastpage}
\usepackage{fancyhdr}
\usepackage[T1]{fontenc}
\usepackage{graphicx}
\usepackage[dvipsnames]{xcolor}
\usepackage{centernot}
\allowdisplaybreaks[1]
\usepackage{tikz-cd}
\usepackage[left= 2.5cm, bottom = 3cm, top = 3.0cm, right= 2.5cm]{geometry}

\numberwithin{equation}{section}
\theoremstyle{plain}
\newtheorem{theorem}{Theorem}[section]        
\newtheorem{corollary}[theorem]{Corollary}             
\newtheorem{lemma}[theorem]{Lemma}                 
\newtheorem{proposition}[theorem]{Proposition}

\theoremstyle{definition}
\newtheorem{definition}[theorem]{Definition}                
\theoremstyle{remark}
\newtheorem{remark}[theorem]{Remark}

\DeclareMathOperator{\Div}{div}
\DeclareMathOperator{\supp}{supp}

\newcommand{\bR}{\mathbb{R}}

\newcommand{\W}{\mathbf{W}}
\newcommand{\BS}{\mathcal{K}}
\newcommand{\st}{ \hspace{0.075cm} \hat{\otimes} \hspace{0.075cm}  }

\makeatletter
\@tfor\next:=abcdefghijklmnopqrstuvwxyzABCDEFGHIJKLMNOPQRSTUVWXYZ\do{%
\def\command@factory#1{%
\expandafter\def\csname b#1\endcsname{\mathbf{#1}}
\expandafter\def\csname cl#1\endcsname{\mathcal{#1}}
}
\expandafter\command@factory\next
}

\newcommand{\tmop}[1]{\ensuremath{\operatorname{#1}}}

\newcommand{\cdummy}{\cdot}

\mathtoolsset{showonlyrefs}

\title{
	On a rough perturbation of the Navier-Stokes system \\ and its vorticity formulation 
}
\author{Martina Hofmanov\'a\thanks{Fakult\"at f\"ur Mathematik, Universit\"at Bielefeld, Postfach 100131, D-33501 Bielefeld, Financial support by the DFG via Research Unit FOR 2402 is gratefully acknowledged.}, \;James-Michael Leahy\thanks{School of Electrical and Computer
Engineering, Cornell University, N.Y., USA}, Torstein Nilssen\thanks{Institute of Mathematics, Technical University of Berlin, Germany, Financial support by the DFG via Research Unit FOR 2402 is gratefully acknowledged.}}
\date{\today}

\begin{document}
\maketitle

\begin{abstract}
We introduce a rough perturbation of the Navier-Stokes system and justify its physical relevance from balance of momentum and conservation of circulation in the inviscid limit.
We present a framework for a well-posedness analysis of the  system. In particular, we define an intrinsic notion of solution based on ideas from the rough path theory and study the system in an equivalent vorticity formulation. In two space dimensions, we prove that well-posedness and enstrophy balance holds. Moreover, we derive rough path continuity of the equation, which yields a Wong-Zakai result for Brownian driving paths, and  show that for a large class of driving signals, the system generates a continuous random dynamical system. In dimension three, the noise is not enstrophy balanced, and we establish the existence of local in time solutions.

\bigskip

MSC Classification Numbers: 60H15, 76D05, 47J30, 60H05, 35A15.

Key words: Rough paths, Stochastic PDEs, Navier-Stokes system, variational method.

\end{abstract}

\tableofcontents

\section{Introduction}
\label{s:intro}

\subsection{General motivation}

In this paper, we investigate  well-posedness and stability of a rough-path perturbation of the Navier-Stokes system. The deterministic Navier-Stokes equations are a system of non-linear partial differential equations that govern the velocity field $u$ and pressure $p$ of an incompressible homogeneous viscous fluid moving in some domain $\clD\subseteq \bR^d$:
\begin{equation}\label{eq:NS}
\begin{aligned}
\partial_t u+(u\cdot \nabla)u&=-\nabla p+\vartheta\Delta u,  \quad (t,x)\in (0, T) \times  \clD,\\\
\nabla \cdot u&=0, \quad 
u|_{t=0} =u_0,
\end{aligned}
\end{equation}
where
$\vartheta$ is the kinematic viscosity, $u_0$ is a given initial velocity and additional boundary conditions are needed depending on the domain $\clD$.  The system \eqref{eq:NS} can be derived from the basic physical principles by assuming conservation of mass and momentum in integral form, homogeneity, incompressibility (or conservation of kinetic energy) and viscous stress forces, and using Reynold's transport theorem. At least formally, the time-dependent vector field $u$ generates a time-homogeneous two-parameter flow $\eta_{s,t}$ 
on $\clD$:
$$
\dot{\eta}_{s,t}(x) = u_t(\eta_{s,t}(x)), \quad \eta_{s,s}(x)=x, \quad s \leq t, \; \, x \in \clD.
$$
That is, a particle initially at a point $x \in \clD$ at  time $s$ moves to the point $\eta_{s,t}(x)\in \clD$ at time $t$ in such way that at each $t' \in (s,t)$, the instantaneous velocity is given by $u_{t'}(\eta_{s,t'}(x))$.

In practice, solutions of the  Navier-Stokes system are numerically approximated. Due to limited computational resources, there are subgrid dynamics or fast modes that can not be resolved by a direct numerical simulation. The non-linear term $(u\cdot \nabla) u$ mixes the subgrid and grid scales. As such, accurate forecasts of turbulent fluid regimes are only possible at the moment if substantial computational resources are invested, which is not a luxury practitioners can afford in real-time applications where data is to be assimilated.  
Lewis Fry Richardson has said: ``Big whirls have little whirls that feed on their velocity, and little whirls have lesser whirls and so on to viscosity." Here, Richardson is describing the direct energy cascade in 3D turbulence, in which energy is transferred from larger eddies (modes) to smaller eddies to the minimum scale at which the energy is dissipated by viscosity. 
In fluid dynamics and turbulence modeling especially, the search for  tractable models for subgrid-scale dynamics that are closable, parameterizable, and preserve physical laws is ongoing (see, e.g., \cite{majda2003systematic} for one such example).  While all parameterization schemes are designed to improve the quality of forecasts,  stochastic parameterization schemes have an additional advantage in that they provide a natural mechanism to quantify uncertainty in prediction. 

An important property of a parameterized dynamical system is the stability of the  dynamics with respect to the parameters.  In order to define stability, one must specify a set of input parameters and an output set (of the dynamical system), and  endow the corresponding sets with a topology. For a parameterized stochastic dynamical  system, there are two main types of stability, which we will briefly explain.
Let  $S$  denote the  output of a parameterized stochastic dynamical system, which takes values in a space $\clN$ and depends on time $t\in \bR_+$, space $x\in \clM$, a set of parameters $\Theta$, and a sample space outcome $\omega\in \Omega$ (where $(\Omega,\clF,\bP)$ is a probability space).  Probabilistic stability usually means continuity of the map
$$
S :\Theta \rightarrow O\subset L^0(\Omega\times \bR_+\times \clM;\clN),
$$
where  $O$ is a metric space contained in $L^0(\Omega\times \bR_+\times \clM ;\clN)$, the space of measurable random variables from $\Omega\times \bR_+\times \clM$ to $\clN$.  Pathwise stability, on the other hand, means  continuity of the map
$$
S :\Omega\times \Theta \rightarrow \tilde{O}\subset L^0( \bR_+\times \clM ;\clN),
$$
where  $\Omega$ shall be  endowed with certain topology.

To study stability in this sense, a solution map needs to be constructed for each $\omega$; in other words, $S(\omega)$ is the outcome of a deterministic dynamical system.  If $(\Omega,\clF,\bP)$ is  the canonical probability space  for a multi-dimensional Wiener process and the model contains a stochastic integral, then, in general, there is no separable Banach space contained in the space of continuous functions $\Omega$ that contains the trajectories of the Wiener process almost-surely and for which the solution map $S$ is pathwise stable. The key idea of rough paths is to consider an  enriched set $\boldsymbol{\Omega}$ of rough paths (i.e., an appropriate feature set for the Brownian paths), which contains additional information beyond the path itself, namely the iterated integrals of the path $\omega$, which one can construct by probabilistic methods. 
The map $S$ is then factorized as follows:
$$
\begin{tikzcd}
\Omega\times \Theta \arrow{dr}{\Phi\times \textnormal{id}} \arrow{rr}{S} & & L^0( \bR_+\times \clM ;\clN),\\
& \boldsymbol{\Omega}\times \Theta \arrow{ur}{\tilde{S}} &
\end{tikzcd}
$$
where  $\Phi$ is a  measurable  feature map which  `lifts' the path to a rough path and  $\tilde{S}$ is a continuous (Lipschitz in some cases) `path-by-path' solution map. The construction of $\tilde{S}$  allows Brownian paths to be treated as a parameter belonging to the set of rough paths, which puts the stochastic and deterministic parameterization schemes on equal footing as far as stability is concerned.

As mentioned above, stochastic parameterization schemes provide a natural mechanism to forecast uncertainty. That is,  an ensemble of solutions can be generated. By constructing a path-by-path solution map $\tilde{S}$, any element of the enriched space $\boldsymbol{\Omega}$ is an admissible driving path. For example, non-Markovian processes such as fractional Brownian motion have rough path lifts to $\boldsymbol{\Omega}$. Thus, a highly flexible stochastic modeling framework is permissible once the pathwise solution map $\tilde{S}$ is constructed. Recent work on the statistics of 2D fluid turbulence suggests that the subgrid (or fast scales) dynamics of fluids are non-Markovian and non-Gaussian \cite{lilly2016fractional, faranda2014modelling}. In fact, even piecewise linearly interpolated data from observations could serve as a driving path.

The system of rough partial differential equations we consider in this paper arise from perturbing the advecting vector field in \eqref{eq:NS} by a time-dependent vector field that is rough in time and smooth in space. More precisely, we re-write \eqref{eq:NS} in covariant form, and then perturb the advecting vector field. The perturbation can be understood as a parameterization of the subgrid dynamics or high-modes of the fluid velocity field.  Therefore, the well-posedness and stability results we establish clear the way for the  development of a rich and robust modeling framework for fluids.

\subsection{Derivation  of the equation} \label{sec:Derivation}

In this section we present a heuristic derivation of our main equation and discuss its physical relevance. However, this is not essential for reading and understanding our results in the remainder of the paper and, as such, may be skipped during the first reading. 

The  Navier-Stokes system \eqref{eq:NS} is the differential form of the momentum balance  principle under the additional assumption that the fluid is homogeneous (constant density) and incompressible. The momentum balance in integral-form and in standard coordinates reads
$$
\frac{d}{dt}\int_{\eta_{s,t}(W)}u_t^i\rho dx = \int_{\eta_{s,t}(W)}\nu \Delta u_t^i dx - \int_{\eta_{s,t}(W)}p n^i dA = \int_{\eta_{s,t}(W)}(\nu \Delta u_t^i -\partial_{x^i} p_t) dx,  
$$
for all nice regions $W\subset \clD$, where we have written the coordinates to emphasize the fact that the momentum balance principle as stated is coordinate dependent.

It is a worthwhile endeavor to derive an equation for the momentum balance that is invariant under a change of the coordinate system (see, e.g., for  \cite{tao2016finite} for motivation and \cite{abraham2012manifolds} or  \cite{taylor2013partial} for more details). The language of differential geometry provides the tools to do so, while also providing the natural generalization of the fluid equations to a  manifold $M$. 

One usually considers the fluid velocity $u$ in \eqref{eq:NS} as a vector field, which we write as $u = u^j \frac{\partial}{ \partial x_j}$. where $(x,U)$ is a local coordinate system of $M$ and $\frac{\partial}{ \partial x_j}$ is the local basis of the tangent bundle $TM$. Here and for the rest of the paper we use the convention of summation over repeated indices. In the inviscid case $\vartheta = 0$, that is, for Euler's equations, the momentum balance principle
 implies conservation of circulation by Reynold's transport theorem: 
\begin{equation} \label{eq:CirculationConservationVectorField}
\oint_{\eta_{s,t}(C)} u_t  = \oint_{C} u_0  
\end{equation}
for any $s,t$ and any contour $C$. The  reader will notice the ambiguity of the above integrals--the contour is a 1-dimensional subset of $M$ and as such one should really understand $u$ as a 1-form. One can obtain a 1-form from $u$ on a Riemannian manifold $(M,g)$ by setting $u^{\flat} := g_{ij} u^j d x^j$, where $dx^j$ is a local basis of the cotangent bundle $T^*M$ and $g_{ij}$ is the metric tensor in local coordinates. To simplify our discussion below, we assume the manifold is flat $g_{ij}=\delta_{ij}$ in what follows. Thus, the contour integrals above can be written as line integrals of the one-form $u^{\flat}$:
\begin{equation} \label{eq:CirculationConservation}
\oint_{\eta_{s,t}(C)} u_t^{\flat}  = \oint_{C} u_0^{\flat} .
\end{equation}

To obtain a coordinate-free expression for $u^{\flat}$, we first consider the Navier-Stokes equation in standard coordinates:
$$
\partial_t u^i+u^j \frac{\partial }{\partial x_j} u^i=-\frac{\partial }{\partial x_i} p+\vartheta\Delta u^i, \;\;i \in \{1,\ldots,d\}.
$$
Adding $u^j\frac{\partial }{\partial x_i} u^j$ to both-sides of the equation, we get 
$$
\partial_t u^i + u^j\frac{\partial }{\partial x_j} u^i + u^j \frac{\partial }{\partial x_i} u^j=-\frac{\partial }{\partial x_i}\tilde{p}+\vartheta\Delta u^i,
$$
where $\tilde{p}= p-\frac{1}{2}|u|^2$. The reason for adding this term to both sides is that the last two terms on the left-hand side of the equality can be identified with the Lie derivative of the one-form $u^{\flat}$ along $u$:
\begin{align*}
\clL_{u_t} u^{\flat}_t & = \frac{d}{d\tau} (\eta_{t,\tau})^* u^{\flat}_t|_{\tau=t}= u^j_t \frac{\partial }{\partial x_j} u^i_t   dx^i +  u^j_t \frac{\partial }{\partial x_i} u^j_t  dx^i,
\end{align*}
where the latter equality  is a direct consequence of  Cartan's magic formula.
Let $\bd$ be the exterior differential operator and $\delta$ the co-differential operator. The operator $\bd\delta +\delta \bd$ is called the Hodge-Laplacian, and is equal to the (Levi-Civita) connection Laplacian on flat space by the Weitzenb\"ock  identity. In particular, 
$$
\Delta u =((\bd\delta +\delta \bd)u^{\flat})^{\sharp}\quad  \Leftrightarrow \quad \Delta u^i=((\bd\delta +\delta \bd)u^{\flat})_i, \quad i \in \{1,\ldots d\},
$$
where $\sharp$ denotes the inverse of the $\flat$ operator. 
Putting it all together,  the covariant form of the Navier-Stokes equation is given by
\begin{equation}\label{eq:covformNS}
\partial_t u^{\flat}+\clL_{u}u^{\flat}=-\bd\tilde{p}+\vartheta \delta \bd u^{\flat}, \quad \delta u^{\flat}=0,
\end{equation}
where the divergence-free condition is written in terms of the codifferential.
The term $\clL_{u}u^{\flat}$ is the non-linear Lie-advection of the one-form $u^{\flat}$ by the vector-field $u$ whose associated flow generates the integral curves $\eta$. As an application of Reynold's transport theorem,  we find
$$
\frac{d}{dt}\oint_{\eta_{s,t}(C)} u_t^{\flat} = \oint_{\eta_{s,t}(C)} (\partial_t+\clL_{u_t})u_t^{\flat}=\oint_{\eta_{s,t}(C)}\left( \bd \tilde{p}+\vartheta(\bd+\delta)^2 u^{\flat}\right),
$$
which, upon applying  Stokes' theorem, gives a convenient proof of circulation conservation when $\vartheta =0$.

 In practice, one must approximate solutions of \eqref{eq:covformNS}, and hence  ignore the high modes of the solution. That is, one  can only compute  solutions of
\begin{equation}
\partial_t u^{L,\flat}+\clL_{u^L}u^{L,\flat}=-\bd\tilde{p}+\vartheta \delta \bd u^{L,\flat}, \quad \delta u^{L,\flat}=0,
\end{equation}
where $u^{L}$ has only modes up to a certain order. A  way of improving approximations on a limited computational budget is to  parameterize  the high-modes $u^{H}$  of $u$ by a vector field $\tilde{u}^{H}$ and compute
\begin{equation}
\partial_t u^{L,\flat}+\clL_{u^L+\tilde{u}^H}u^{L,\flat}=-\bd\tilde{p}+\vartheta \delta \bd u^{L,\flat}, \quad \delta u^{L,\flat}=0.
\end{equation}
One possible choice of a parameterization of $u^{H}$ is given by $\tilde{u}^{H}=\sigma_k \dot{B}_t^k$, where $\sigma_{k}:M\to\bR^{d}$, $k\in \{1,\dots, K\}$, are sufficiently regular divergence-free vector fields and $B^{k}:\bR_{+}\to \bR$ are independent Brownian motions. The idea is that $u$ is approximated by a stochastic ensemble of solutions of $u^{L}$. In fact, such an equation can be derived from the theory of stochastic homogenization combined with a variational principle, and we refer the reader to \cite{cotter2017stochastic, cotter2019numerically} for more details about the derivation and for verifiable proof that the parameterization is flexible enough to capture the high-modes of $u$. Motivated by the practical success of this approach, we seek to develop a framework for more flexible parameterizations, where instead of Brownian motions $B^k$, one considers rough paths $z^k$, and to develop pathwise stability of the Brownian case, at least in dimension two.

Motivated by this problem, we perturb the advecting vector field $u$ in $\clL_{u}u^{\flat} $  in \eqref{eq:covformNS} by a random vector field of the form $\sigma_k \dot{z}^k$, where $\sigma_{k}:M\to\bR^{d}$, $k\in \{1,\dots, K\}$, are sufficiently regular divergence-free vector fields and $z^{k}:\bR_{+}\to \bR$ are driving paths, which shall eventually possess only a limited regularity.  That is, we replace $\clL_{u}u^{\flat}$ with $ \clL_{u+\sigma_k \dot{z}^k}u^{\flat}$ and consider 
\begin{equation} \label{eq:MainEquationNoCoordinates}
\partial_t u^{\flat}+\clL_{u+\sigma_k\dot{z}^{k}}u^{\flat}=\partial_t u^{\flat}+\clL_{u}u^{\flat}+\clL_{\sigma_k}u^{\flat} \dot{z}^k_t =-\bd\tilde{p}+\vartheta\delta \bd u^{\flat}.
\end{equation}
The vector field $u+\sigma_k \dot{z}^k$ generates the two-parameter flow $\tilde{\eta}$ on $M$: 
$$
\dot{\tilde{\eta}}_{s,t}(x) = u_t(\tilde{\eta}_{s,t}(x)) +\sigma_k(\tilde{\eta}_{s,t})\dot{z}^k_t, \quad \tilde{\eta}_{s,s}(x)=x, \quad s \leq t, \; \, x \in M.
$$
We understand this on a formal level, since   it is not clear how to construct the flow map $\tilde{\eta}$ due to the low regularity. 
Applying Reynold's transport theorem,  we find
$$
\frac{d}{dt}\oint_{\tilde{\eta}_{s,t}(C)} u_t^{\flat} = \oint_{\eta_{s,t}(C)} (\partial_t+\clL_{u+\sigma_k\dot{z}^{k}})u_t^{\flat}=\oint_{\eta_{s,t}(C)}\left( \bd \tilde{p}+\vartheta(\bd+\delta)^2 u^{\flat}\right),
$$
which yields conservation of circulation in the inviscid case $\vartheta=0$ (see also, \cite{crisan2017solution}).

Writing \eqref{eq:MainEquationNoCoordinates} in local coordinates, we obtain 
\begin{equation} \label{eq:MainEquationIntro}
\partial_t u^i  + u^j \frac{\partial}{\partial x_j} u^i + \left[ \sigma_k^j \frac{\partial}{\partial x_j} u^i + u^j \frac{\partial}{\partial x_i} \sigma^j_k \right] \dot{z}_t^k = - \frac{\partial}{\partial x_i} p + \vartheta \Delta u^i,
\end{equation}
where we note that we are again writing $p$ (and not 
$\tilde{p}$) which explains that $u^j \frac{\partial}{\partial x_i} u^j$ does not appear in the equation. This is the main equation we  study in this paper. In particular, we introduce a  formulation of the equation well suited  to make sense of the distributional terms $\dot{z}^k$ and to study well-posedness.

However, for technical reasons related to the noise term $\dot{z}^k$, the non-local nature of the pressure term (which translates to the divergence-free condition) makes it difficult to obtain a priori estimates directly from this formulation. We elaborate on this issue a bit more in Section \ref{sec:Related}.

One way to circumvent dealing with the pressure is to consider the 2-form $\tilde{\xi} = \bd u^{\flat}$, called the vorticity. Taking the exterior derivative in \eqref{eq:MainEquationNoCoordinates} and using that $\bd$ commutes with the Lie derivative, we get
$$
\partial_t \tilde{\xi}+\clL_{u}\tilde{\xi} + \clL_{\sigma_k} \tilde{\xi}  \dot{z}^k_t =\vartheta \bd \delta   \tilde{\xi}.
$$
To write the above in local coordinates, we consider the Hodge dual of the vorticity, which we denote by $\xi$, and is equal to the scalar $\ast \tilde{\xi}$ in dimension two and  the vector field $(\ast \xi)^{\sharp}$ in dimension three, where $\ast$ is the Hodge-star operator, which maps $2$-forms to $d-2$-forms. It follows that (see, e.g.,  pages 451 and 566 in \cite{taylor2013partial} and recall that we have assumed flatness for simplicity) 
$$
\partial_t \xi+\clL_{u}\xi + \clL_{\sigma_k} \xi  \dot{z}^k_t =\vartheta \Delta \xi,
$$
where $\clL_{\sigma_k} \xi=\sigma_k(\xi)$ in dimension two  since $\xi$ is a scalar and  $\clL_{\sigma_k} \xi =[\sigma_k,\xi]$ in dimension three since $\xi$ is a vector field, and we have slightly abused notation in writing the Laplacian on the right-hand-side. In standard coordinates, $\xi$ solves a scalar transport equation  in dimension two:
$$
\partial_t \xi + u^j \frac{\partial}{\partial x_j} \xi + \sigma_k^j \frac{\partial}{\partial x_j} \xi \, \dot{z}_t^k = \vartheta \Delta \xi, 
$$
 and $\xi$ solves a perturbed version of the usual vorticity equation in dimension three:
$$
(\partial_t \xi +[u+\sigma_k\cdot{z}^k,\xi])^i=\partial_t \xi^i + u^j \frac{\partial}{\partial x_j} \xi^i - \xi^j \frac{\partial}{\partial x_j} u^i + \left[ \sigma_k^j \frac{\partial}{\partial x_j} \xi^i - \xi^j \frac{\partial}{\partial x_j} \sigma_k^i \right]  \dot{z}_t^k = \vartheta \Delta \xi^i , \quad i \in \{ 1,2,3\}.
$$
The reader will notice that the difference between $d=2$ and $d=3$ is the presence of the formidable vorticity stretching terms $\xi^j \frac{\partial}{\partial x_j} u^i$ and $\xi^j \frac{\partial}{\partial x_j} \sigma_k^i$ in $d=3$, whose presence causes  difficulty from the analytic point of view, but interesting dynamics from the modeling point of view.
For convenience (and with a slight abuse of notation) we  abbreviate the two  equations for $\xi$ as
\begin{equation} \label{eq:MainEqVorticity}
\partial_t \xi + (u \cdot \nabla)  \xi - 1_{d=3}( u\cdot \nabla)\xi + \left[ (\sigma_k \cdot \nabla) \xi  - \mathbf{1}_{d=3} (\sigma_k\cdot \nabla) \xi\right] \dot{z}_t^k = \vartheta \Delta \xi ,
\end{equation}
and we note that there is no non-locality, meaning no pressure term which would influence the noise. In dimension two, by formally testing against $\xi$ and using the fact that the $\sigma_k$ are divergence-free, we find
\begin{equation}\label{eq:enst est intro}
|\xi_t|_{L^2}^2+2\vartheta \int_0^t|\nabla \xi_r|_{L^2}^2 dr=|\xi_0|_{L^2}^2,
\end{equation}
which implies that enstrophy is balanced in dimension two.

Equation \eqref{eq:MainEqVorticity} is, of course, still non-linear due to the presence of $\clL_u$, so one needs to write $u$ in terms of $\xi$. This operation, which is called the Biot-Savart law and acts as an inverse of $\bd$, can only be done up to additive constants since $\bd$ is a derivative. 
More precisely, one can see that the missing constant is the spatial average of $u$, which by formally integrating \eqref{eq:MainEquationIntro} in space, should satisfy 
\begin{equation} \label{eq:AverageEquationDifferentialForm}
\partial_t \int_{M} u^i_t(x) dx   + \int_{M} u^j_t(x) \frac{\partial}{\partial x_i} \sigma^j_k(x)   dx  \, \dot{z}_t^k = 0 
\end{equation}
when we assume $u$ and $\sigma_k$ are divergence-free and under appropriate boundary conditions on $M$. Notice that there is no geometric ambiguity in the above spatial integrals since we are considering the components of $u$.

Throughout our analysis, it is therefore necessary to preserve the information in \eqref{eq:AverageEquationDifferentialForm} as it  allows us to recover the full velocity. In other words, we  solve \eqref{eq:MainEqVorticity} and \eqref{eq:AverageEquationDifferentialForm} as a system of equations, which  is  better suited for deriving a priori estimates of \eqref{eq:MainEqVorticity}, which from now on will be referred to as \emph{enstrophy} estimates. In addition, the system  \eqref{eq:MainEqVorticity}, \eqref{eq:AverageEquationDifferentialForm}  is  shown  to be equivalent to \eqref{eq:MainEquationIntro} under the condition $\nabla\cdot u=0$. 
We note that there this issue does not appear  in the classical Navier-Stokes equations, that is, in the case $z^k = 0$. Indeed, equation \eqref{eq:AverageEquationDifferentialForm} shows that the Navier-Stokes system conserves the spatial average so that one may without loss of generality assume that $\int_M u_0^i(x) dx = 0$.

\subsection{Related literature and main contributions}
\label{sec:Related}

The stochastic Navier-Stokes equation has been well studied using Brownian motion as the driving noise. With no ambition  at an exhaustive list of references, let us  mention \cite{brzezniak1991stochastic}, \cite{brzezniak1992stochastic}, \cite{mikulevicius2004stochastic}, \cite{mikulevicius2005global} and \cite{flandoli1995martingale}. Moreover, a similar multiplicative noise as in the present paper have been studied in \cite{crisan2017solution} and \cite{BFM}. 
In the pathwise setting, using regularity structures, the Navier-Stokes system with space-times white noise has been studied in \cite{ZHU20154443}.

A problem similar to \eqref{eq:MainEquationIntro}, namely,
\begin{align}
\partial_t u  + (u \cdot \nabla )u & +  (\sigma_k \cdot \nabla)  u\, \dot{z}_t^k = - \nabla p + \vartheta \Delta u \label{OldEquation} \\
\nabla \cdot u&=0, \quad 
u|_{t=0} =u_0 \in \bL^2, \notag
\end{align}
has been studied by the same authors in \cite{HLN}.
On the surface, the main difference between \eqref{eq:MainEquationIntro} and \eqref{OldEquation} is that the noise in \eqref{OldEquation} is energy conservative for the velocity. However, based on the discussion in Section~\ref{sec:Derivation}, we see  that,  in general, the perturbation does not conserve circulation in the inviscid case nor enstrophy balance in dimension two. In fact, \eqref{OldEquation} is usually obtain by treating the solution of Navier-Stokes as a collection of scalar equations, thus ignoring the geometry of the problem (i.e., the Lie derivative). 
Furthermore, there are deep, highly technical and structural reasons why energy conservation for the velocity does not yield satisfactory well-posedness results.

More precisely, as it will become clear in the derivation below in Section \ref{ss:form}, applying the Helmholtz projection to the equation  entangles a non-locality into the rough integral term.
The only available method\footnote{An alternative method has been introduced in \cite{HNS}, but it is not clear whether it is applicable to  systems or how to encode the divergence-free condition.} to obtain uniqueness of weak solutions for rough PDEs in the variational setting is the method introduced in \cite{BaGu15}, based on commutator estimates \`a la DiPerna and Lions \cite{DiPernaLions}.
However, this approach seems to fail under the presence of the Helmholtz projection.
Consequently,  uniqueness in \cite{HLN} could only be proved under very  restrictive assumptions on the vector fields $\sigma_{k}$, namely the ones for which the rough term commutes with the Helmholtz projection, effectively restricting to constant vector fields.

Leaving this aside, there is also a structural problem with the equation containing the projection, even if one \emph{could} use the techniques of \cite{DiPernaLions}. Indeed, since the Helmholtz projection is not continuous on $L^{\infty}$, the equation for  $uu^{T}$, which is needed for the energy estimates, contains noise that cannot be made sense of in an appropriate Banach space as for instance $(L^{\infty})^{*}$ that is dictated by the deterministic part of the equation. To summarize, it is the unfavorable interplay between the energy conservative noise and the deterministic part of the equation reflected through the Helmholtz projection, which makes the problem not easily accessible for a direct  pathwise analysis.

\medskip

In the present paper, we take a different path and develop a model in which the noise conserves circulation in the inviscid limit (which we do not address in this paper) and enstrophy balance in dimension two. Note that the enstrophy corresponding to the $L^{2}$-norm of the vorticity is balanced in two space dimensions and conserved  in case $\vartheta=0$ as in the deterministic unforced setting. Enstrophy, however,  is not balanced in three dimensions, which leads to a significantly more involved analysis and only local in time solutions.
Moreover, since the vorticity formulation eliminates the pressure, the non-locality of the equation for the vorticity does  not influence the noise-term.
On the other hand, as discussed above, particular care has to be taken in order to fully recover the velocity from the vorticity formulation on the Torus. This subtlety seems to have been missed in the available literature.

Additionally, we establish  pathwise continuity properties which  easily enables the study of Wong-Zakai results for the case of Brownian motion with Stratonovich integration. Furthermore,  a generation of  a continuous random dynamical system follows for a large class of driving stochastic processes. This set up could also be used for studying large deviations and support theorems. 

Another contribution of this paper is that it further develops the theory of  unbounded rough drivers as introduced in \cite{BaGu15} and further developed in \cite{DeGuHoTi16} - a method aimed at studying PDEs perturbed by an unbounded operator valued noise term. Still, an abstract variational method in the spirit of \cite{Rockner} for these equations is not available. 
We believe that the way the present paper tailors the method of \cite{BaGu15} to the Navier-Stokes equation in a non-trivial way could explain why this is the case.
On the other hand, exactly this fact of being able to tune the method demonstrates its flexibility and suggests that this mentality could be used for studying other equations, and possibly build towards a general theory.

The paper is organized as follows. Section \ref{s:setting} is devoted to notations and definitions. The precise formulation of the problem, derivation of the vorticity formulation and the main results are described in Section \ref{ss:form}. Section \ref{sec:AprioriEst} contains basic a priori estimates. Enstrophy balance, uniqueness as well as the rough path stability and Wong-Zakai result in two space dimensions is presented in Section \ref{section:ProofOfTheoremUniqueness}, whereas Section \ref{Section:Galerkin} contains the proof of existence. Certain auxiliary results are collected in the appendix.

\section{Preliminaries}
\label{s:setting}

In this section, we introduce the notation and collect the basic definitions needed in the sequel.

\subsection{Sobolev spaces and vector calculus}
\label{ss:notation}

We begin by fixing the notation that we use throughout the paper.

For  a given $d\in \mathbf{N}$, let  $\bT^d=\bR^d/(2\pi \bZ)^d$ be the $d$-dimensional flat torus and denote by $dx$ the unnormalized Lebesgue measure on $\bT^d$. As usual, we blur the distinction between periodic functions and functions defined on the torus $\bT^d$. 
Let $\bL^2=L^2( \bT^d ; \bR^d)$.
For a given $m\in \bZ$, we denote by $\W^{m,2}$ the Sobolev space $\W^{m,2}=(I-\Delta)^{-\frac{m}{2}}\bL^2$ and by $|\cdot|_{m}$ the corresponding norm.
Let $\dot{\W}^{m,2}$ be the subspace of mean-free vector fields; that is, $f \in \W^{m,2}$ such that $\bar{f}:=\int_{\bT^d}f(x) dx = 0$. Although all these spaces contain vector valued functions, we shall sometimes abuse notation and write $\bW^{2,m}$ etc also for scalar functions when it is clear from the context. 
The Leray projection onto divergence-free vector fields is denoted by $P$
and we let  $Q = I - P$.  
We set 
$
\bH^{m}=P\W^{m,2}$ and  $\bH_{\perp}^{m}=Q\W^{m,2}
$  
and recall that  for all $m\in \bZ$ (see Lemma 3.7 in \cite{mikulevicius2003cauchy}),
$
\W^{m,2} =\bH^m\oplus \bH_{\perp}^m,
$
where 
$$
\bH^m=\left\{f\in \W^{m,2}: \;\nabla \cdot  f= 0\right\},
\quad \textrm{ and } \quad 
\bH_{\perp}^m=\{g\in \W^{m,2}:  (f,g)=0, \;\; \forall f\in \bH^{-m} \}.
$$
The corresponding mean-free spaces will be denoted $\dot{\bH}^m$ and $\dot{\bH}^m_{{\perp}}$.

\medskip

We will abbreviate notation and write simply $\partial_i$ denoting the derivative $\frac{\partial}{\partial x_i}$. Let $\sigma:\bT^d\rightarrow \bR^d$ be twice differentiable and assume that the derivatives of $\sigma$ up to order two are bounded uniformly by a constant $N_0$.
Let  $\clA^1=\sigma\cdot \nabla=\sum_{i=1}^d\sigma^i \partial_i$ 
and  $\clA^2=(\sigma\cdot \nabla)(\sigma\cdot \nabla).$
It follows  that there is a  constant  $N=N(d,N_0,\alpha)$  such that 
$$
|\clA^1|_{\clL(\W^{m+1,2},\W^{m,2})}\le N, \; m = 0,1,2,\quad
|\clA^2 f|_{\clL(\W^{m+2,2},\W^{m,2})}\le N, \; m = 0,1.
$$
Since  $P\in \clL(\W^{m,2},\bH^m)$ and $Q\in \clL(\W^{m,2}, \bH^m_{\perp})$ for all $m\in \bZ$,  both of which have operator norm bounded by $1$, we have 
\begin{align} \label{A1bound}
|P\clA^1|_{\clL(\bH^{m+1,2},\bH^{m,2})}\le N, &  \; m = 0,1,2,\quad
|P\clA^2 f|_{\clL(\bH^{m+2,2},\bH^{m,2})}\le N, \; m = 0,1. \\
|Q\clA^1|_{\clL(\bH_{\perp}^{m+1,2},\bH_{\perp}^{m,2})}\le N, & \; m = 0,1,2,\quad
|Q\clA^2 f|_{\clL(\bH_{\perp}^{m+2,2},\bH_{\perp}^{m,2})}\le N, \; m = 0,1.
\end{align}

For a given $u: \bT^d\rightarrow \bR^d$ and $d=2$ or $d=3$, we notice that the exterior derivative coincides with the curl operator and can be written in standard coordinates as
\begin{align*}
\operatorname{curl} u=\nabla \times u &= \big(\partial_2u^3-\partial_3u^2, \partial_3u^1-\partial_1u^3,\partial_1u^2-\partial_2u^1 \big),\quad \textnormal{ if } d=3,\\
\operatorname{curl} u=\nabla \times u& = \partial_1u^2-\partial_2u^1, \quad \textnormal{ if } d=2.
\end{align*}
We notice that $\nabla \times P = \nabla \times$.

For a mean-free $f:\bT^2\rightarrow \bR$, we define the Biot-Savart operator (i.e., inverse of the curl)
$$
\operatorname{\BS}f
=\nabla^{\perp}(-\Delta)^{-1}f
$$
where we have defined $
\nabla^{\perp}\psi=\partial_2\psi\bi- \partial_1\psi\bj
$
for
$\psi:\bT^2\rightarrow \bR$. For mean-free $f:\bT^3\rightarrow \bR^3$ the Biot-Savart operator is defined as
$$
\operatorname{\BS}f 
=\nabla \times (-\Delta)^{-1}f.
$$
It follows that for mean-free $f$:
\begin{align*}
\nabla \times (-\Delta)^{-1}\operatorname{\BS}f&=(-\Delta)^{-1}f, \;\;  \textnormal{ if }  d=2\\
\operatorname{\BS}^2f&=(-\Delta)^{-1}f, \;\;  \textnormal{ if }  \nabla \cdot f=0 \textnormal{ and } d=3
\end{align*}
and we have
\begin{equation}\label{ineq:BSGrad}
|\nabla  \operatorname{\BS}f|_{m}=| f|_{m}.
\end{equation}
Moreover, for all $n\in \bN$,
$$
\operatorname{\BS}\in \clL(\dot{\bW}^{2,n-1}, \dot{\bH}^{n}), \quad 
\operatorname{curl}\in \clL(\bW^{2,n}, \dot{\bH}^{n-1}),
$$ 
and we have
$$
\operatorname{curl}\circ \operatorname{\BS}\in \clL(\dot{\bW}^{2,n-1}, \dot{\bH}^{n-1})
$$
is the identity operator  for $d=2$ and restricts to the identity operator on $\dot{\bH}^{n-1}$ for $d=3$ and $$\operatorname{\BS}\circ \operatorname{curl}\in \clL(\bW^{2,n},\dot{\bH}^n) $$ restricts to the identity on $\dot{\bH}^{n}$ for $d\in \{2,3\}$.  

\medskip

In order to analyze the non-linear term in \eqref{eq:MainEquationIntro}, we employ the classical notation and bounds.
Owing to Lemma~2.1 in \cite{RT83}, the trilinear form
$$
b(u,\varv,w)= \int_{\bT^d} ((u\cdot \nabla)\varv)\cdot w \, \, dx=\sum_{i,j=1}^d\int_{\bT^d} u^i D_i\varv^j w^j \,\, dx
$$
satisfies the continuity property
\begin{equation}\label{trilinear form estimate}
|b(u,v,w)|\lesssim_{m_1,m_2,m_3,d}
|u|_{m_1}|v|_{m_2+1}|w|_{m_3}, \qquad m_1 + m_2 + m_3 > \frac{d}{2} . 
\end{equation}
Moreover, for all $u\in  \bH^{m_1}$ and  $(\varv,w)\in \bW^{m_2+1,2}\times \bW^{m_3,2}$ such that $m_1,m_2,m_3$ satisfy \eqref{trilinear form estimate}, we have
\begin{equation}\label{eq:B prop}
b(u,\varv,w)=-b(u,w,\varv) \quad \textnormal{and} \quad b(u,\varv,\varv)=0.
\end{equation}
For $m_1,m_2,$ and $m_3$ that satisfy \eqref{trilinear form estimate} we get a bilinear mapping $B: \bW^{m_1,2}\times \bW^{m_2+1,2} \rightarrow \bW^{ - m_3}$ by setting $(B(u,v),w) = b(u,v,w)$.
We define $B_P=PB$ and $B_Q=QB$ and giving the continuous bilinear mappings 
$$
B_P : \bW^{m_1,2}\times \bW^{m_2+1,2}\rightarrow \bH^{-m_3},\quad B_Q: \bW^{m_1,2}\times \bW^{m_2+1,2}\rightarrow \bH^{-m_3}_{\perp}.
$$
We set $B(u) = B(u,u)$ and similarly for $B_P$ and $B_Q$.

\subsection{Rough paths} \label{ss:rp}

For an interval $I,$ we use the notation $\Delta_I := \{(s,t)\in I^2: s\le t\}$ and $\Delta^{(2)}_I := \{(s,\theta,t)\in I^3: s\le\theta\le t\}$.  For simplicity we let $\Delta_T := \Delta_{[0,T]}$ and $\Delta^{(2)}_T=\Delta^{(2)}_{[0,T]}$ for $T>0$. Let $E$ be a Banach space with norm $| \cdot |_E$. 
A function $g: \Delta_I \rightarrow E$ is said to have finite $p$-variation for some $p>0$ on $I$ if 
$$
|g|_{p-\textnormal{var};I;E}:=\sup_{(t_i)\in \clP(I)}\left(\sum_{i}|g _{t_i t_{i+1}}|^p_E\right)^{\frac{1}{p}}<\infty, 
$$
where $\mathcal{P}(I)$ is the set of all partitions of $I$. We denote by $C_2^{p-\textnormal{var}}(I; E)$ the set of all continuous functions with finite $p$-variation on $I$ equipped with the seminorm $|\cdot |_{p-\textnormal{var}; I;E}$ and by $C^{p-\textnormal{var}}(I; E)$ the set of all paths $z : I \rightarrow E$ such that $\delta z \in C_2^{p-\textnormal{var}}(I; E)$, where  $\delta z_{st} := z_t - z_s$. In this section, we drop the dependence of norms  on the space $E$ when convenient. 

A continuous mapping $\omega: \Delta_I \rightarrow [0,\infty)$ is called a control on $I$ provided $\omega(s,s) = 0$ and it is superadditive, namely
$$
\omega(s,\theta)+\omega(\theta ,t)\le \omega(s,t) , \qquad  s \leq \theta \leq t.
$$
If for a given $p > 0$, $g \in C^{p-\textnormal{var}}_2(I; E)$, then  it can be shown that the 2-index map $\omega_g: \Delta_I \rightarrow [0,\infty)$ defined by
$$
\omega_g(s,t)= |g|_{p-\textnormal{var};[s,t]}^p 
$$
is a control  (see, e.g.,  Proposition 5.8 in \cite{FrVi10}). Moreover, it is straightforward to check that one could equivalently define the semi-norm  on $C_2^{p-\textnormal{var}}(I;E)$ by 
\begin{equation}\label{equivdefpvar}
|g |_{p-var; [s,t]} = \inf \{ \omega(s,t)^{\frac{1}{p}} : |g_{uv}| \leq \omega(u,v)^{\frac{1}{p}} \textnormal{ for all } (u, v)  \in \Delta_{[s,t]} \} . 
\end{equation}
We shall need the following local version of the $p$-variation spaces.

\begin{definition}\label{def:variationSpace}
Given an interval $I=[a,b]$, a control $\varpi$ and real number $L> 0$, we denote by $C^{p-\textnormal{var}}_{2, \varpi, L}(I; E)$  the space of continuous two-index maps $g : \Delta_I \rightarrow E$ for which  there exists a control $\omega$ such that for every $(s,t)\in \Delta_I$ with  $\varpi(s,t) \leq L$, it holds that
$
|g_{st}|_E \leq \omega(s,t)^{\frac{1}{p}}.
$ We define a semi-norm on this space by
$$
|g |_{p-\textnormal{var}, \varpi,L; I} =\inf \left\{\omega(a,b)^{\frac{1}{p}} :  \omega  \textnormal{ is a control s.t. } |g_{st}| \leq \omega(s,t)^{\frac{1}{p}}, \;\forall (s, t)  \in \Delta_{I} \textnormal{  with } \varpi(s,t) \leq L \right\} . 
$$ 
\end{definition}
It is clear that
\begin{equation} \label{inclusionOfLocalVariation}
C^{p-\textnormal{var}}_{2, \varpi_1, L_1}(I; E) \subset C^{p-\textnormal{var}}_{2, \varpi_2, L_2}(I; E)
\end{equation}
for $\varpi_1 \leq \varpi_2$ and $L_2 \leq L_1$.

Next, we present the definition of a rough path and the reader is referred  to \cite{MR2314753, FrVi10, FrHa14} for a thorough exposition of the theory of rough paths. For a two-index map $g: \Delta_{I}\rightarrow \bR$, we define the second order increment operator $$\delta g_{s \theta t} = g_{st} - g_{\theta t} - g_{s \theta}, \qquad  s \leq \theta \leq t.$$

\begin{definition}\label{defi-rough-path}
Let $K \in\bN $ and $p\in[2,3)$. A continuous $p$-rough path is a  pair 
\begin{equation}\label{p-var-rp}
\bZ=(Z, \mathbb{Z}) \in C^{p-\textnormal{var}} ([0,T];\bR^{K}) \times C^{\frac{p}{2}-\textnormal{var}}_2 ([0,T]; \bR^{K\times K}) 
\end{equation}
that satisfies the Chen's relation 
\begin{equation*}\label{chen-rela}
\delta \mathbb{Z}_{s\theta t}=Z_{s\theta} \otimes Z_{\theta t},   \qquad  s \leq \theta \leq t.
\end{equation*}
We will denote by $\omega_Z$ the smallest control dominating both $|Z_{st}|^p$ and $|\mathbb{Z}_{st}|^{\frac{p}{2}}$. 
A continuous $p$-rough path $\mathbf{Z}=(Z, \mathbb{Z})$ is said to be geometric if it can be obtained as the limit in the   product topology $C^{p-\textnormal{var}}_2 ([0,T];\bR^{K}) \times C^{\frac{p}{2}-\textnormal{var}}_2 ([0,T]; \bR^{K\times K})$  of a sequence of rough paths  $\{(Z^{n},\mathbb{Z}^{n})\}_{n=1}^{\infty}$ such that for each $n\in \bN$, 
$$Z^{n}_{st}:= \delta z^{n}_{st} \quad \textnormal{ and } \quad \mathbb{Z}^{n}_{st}:=\int_s^t \delta z^{n}_{s r} \otimes \dot{z}^n_{r} d r ,$$
for some smooth path $z^n:[0,T] \to \bR^K$.
We denote by $\mathcal{C}^{p-\textnormal{var}}_g([0,T];\bR^K)$ the set of geometric $p$-rough paths and endow it with the product topology.
\end{definition}

\subsection{Unbounded rough drivers}\label{sec:unboundedrough}
In \cite{Davie}, A.M. Davie makes the groundbreaking observation that rough differential equations can be interpreted as an equation in Taylor expansions. This notion of solution is obtained by iterating a rough differential equation into itself and using Taylor's formula to re-expand non-linearities in terms of the equation itself. The final expression is an increment equation that allows for detailed analysis of the solution in terms of the oscillations of the temporal noise.

Copying this to the framework of PDEs with unbounded perturbations, we are led to iterating the vector fields acting on the solution. Now, the oscillations in time are coupled with spatial derivatives.
Thus one needs  appropriate function spaces in order  to capture the behavior of the involved quantities with respect to the spatial variable. 
	
In what follows, consider a quadruple $(E^{n}, | \cdot |_n)_{ n =0}^3$ of Banach spaces such that $E^{n + k}$ is continuously embedded into $E^{n}$ for $k,n\in\{0,\dots,3\}$ such that $n+k \leq 3$. We denote by $E^{-n}$ the topological dual of $E^{n}$, and note that, in general, $E^{-0}\neq E^0.$ When the norm is clear from the context, we call $(E^n)_n$ a scale of spaces and it is understood that $n\in\bN$ such that $-3 \leq n \leq 3$.

\begin{definition}
\label{def:urd}
Let $p\in [2,3)$ and $T>0$ be given. A continuous unbounded $p$-rough driver with respect to the scale $(E^n)_{n}$, is a pair $\mathbf{A} = (A^1,A^2)$ of $2$-index maps such that
there exists a continuous control $\omega_A$ on $[0,T]$ such that for every $(s,t)\in \Delta_T$,
\begin{equation}\label{ineq:UBRcontrolestimates}
| A^1_{st}|_{\mathcal{L}(E^{n},E^{n - 1})}^p \leq\omega_{A}(s,t) \ \  \text{for}\ \ -2 \leq n \leq 3 , \quad
|A^2_{st}|_{\mathcal{L}(E^{n},E^{n-2})}^{\frac{p}{2}} \leq\omega_{A}(s,t) \ \ \text{for}\ \ 0\leq n\leq 2,
\end{equation}
and   Chen's relation holds true,
\begin{equation}\label{eq:chen-relation}
\delta A^1_{s\theta t}=0,\qquad\delta A^2_{s\theta t}= A^1_{\theta t}A^1_{s\theta},\;\;\forall (s,\theta,t)\in\Delta^{(2)}_T.
\end{equation}
\end{definition}
We shall need a tool that allows us to compare the regularity of the different spaces in the scale $(E^n)_{n}$.

\begin{definition} \label{def:smoothingOp}
A family of smoothing operators $(J^{\eta})_{\eta \in (0,1]}$ acting on $(E^n)_{n }$ is a family of self-adjoint operators such that,
\begin{equation} \label{smoothingOperator}
|(I - J^{\eta}) f |_{n} \lesssim \eta^{k} |f|_{n+k} \hspace{.5cm} \textnormal{ and} \hspace{.5cm} | J^{\eta} f |_{n+k} \lesssim \eta^{-k} |f|_{n},
\end{equation}
for $0 \leq k \leq 2$ and $-2 \leq n \leq 2$ such that $-3 \leq n+k \leq 3$. 
\end{definition}
In the scale $(\bH^{n})_{n }$ a family of smoothing operators can be constructed using the frequency cut-off, see \cite{HLN}. In fact, in this case \eqref{smoothingOperator} is valid for any integers $k,n$. In the case of $L^{\infty}$-scale on the torus, one may employ convolution with a nonnegative smoothing kernel to obtain \eqref{smoothingOperator}  for $k\in \{0,1,2\}$. 

We include here the main a priori estimate from \cite{DeGuHoTi16}. See however Theorem \ref{Thm2.5} for a related result. 
\begin{theorem} \label{Thm2.5Abstract}
Assume
\begin{itemize}
\item
$(E^n)_{n }$ is a scale of spaces for which there exists a family of smoothing operators;

\item
$\mathbf{A}  = (A^1,A^2)$ is an unbounded $p$-rough driver on $(E^n)_{n }$ for $p \in [2,3)$;

\item
$\mu: I \rightarrow E^{-1}$ is of bounded 1-variation (i.e., $| \delta \mu_{st} |_{-1} \leq \omega_{\mu}(s,t)$ for some control $\omega_{\mu}$);
\item
$g: I \rightarrow E^{-0}$ is a bounded path such that 
\begin{equation} \label{eq:AbstractURDEquation}
d g_t =  d \mu_t + \mathbf{A}(dt) g_t
\end{equation}
in the sense that 
$$
g_{st}^{\natural} := \delta g_{st} - \delta \mu_{st} - A_{st}^1 g_s - A_{st}^2 g_s
$$
belongs to $C_{2,\varpi,L}^{\frac{p}{3} - var}(I;E^{-3})$.
\end{itemize}
Then there exists a constant $\tilde{L}$ such that for all $(s,t) \in \Delta_I$ with $\varpi(s,t) \leq L$ and $\omega_A(s,t) \leq \tilde{L}$ we have
$$
|g_{st}^{\natural}|_{-3} \lesssim |g|_{L^{\infty}_I E^{-0}} \omega_A(s,t)^{\frac{3}{p}} + \omega_{\mu}(s,t) \omega_A(s,t)^{\frac{1}{p}} .
$$
\end{theorem}

\section{Formulation of the problem and the main results}
\label{ss:form}

As the first step of our analysis, we derive a rough path formulation of \eqref{eq:MainEquationIntro} and \eqref{eq:MainEqVorticity}, which will be satisfied by solutions constructed by our main existence result below, Theorem \ref{existenceThmNoGalerkin}. For notational convenience, we change the sign of the vector fields $\sigma_k$, so that \eqref{eq:MainEquationIntro} becomes
\begin{equation}
\begin{aligned} \label{eq:RNSDiffForm}
\partial_t u_t + (u_t \cdot \nabla) u_t +\nabla p_t    & =   \vartheta \Delta u_t+[(\sigma_k\cdot \nabla )u_t +(\nabla \sigma_k)u_t]\,\dot{z}^k_t ,  \\
\nabla\cdot u_t & = 0, \quad  u_t|_{t=0}  = u_0,
\end{aligned}
\end{equation}
for a given initial condition $u_0\in \bH^1$ and vector fields $\sigma_k \in \W^{3, \infty}_{\Div}$ (i.e., divergence-free).  The unknowns in  \eqref{eq:RNSDiffForm} the velocity field $u: [0,T] \times \bT^d \rightarrow \bR^d$ and the  pressure $p: [0,T] \times \bT^d \rightarrow \bR$.

We study the Navier-Stokes equation in the variational framework by decoupling the velocity field and the pressure into two equations using the Leray projection $P$ defined in Section \ref{ss:notation}. Applying the solenoidal $ P :\W^{m,2} \rightarrow \bH^m$ and gradient projection $Q: \W^{m,2} \rightarrow \bH^{m}_{\perp}$ separately to \eqref{eq:RNSDiffForm} yields
\begin{equation} \label{eq:RNSDiffFormSys}
\begin{aligned}
\partial_t u_t + P[ (u_t\cdot \nabla) u_t]&= \vartheta \Delta u_t+P[(\sigma_k\cdot \nabla) u_t+(\nabla \sigma_k)u_t] \dot{z}_t^k,  \\
\nabla p_t+Q[ (u_t \cdot \nabla)  u_t] &=Q [(\sigma_k\cdot \nabla) u_t+(\nabla \sigma_k)u_t]  \dot{z}_t^k.
\end{aligned}
\end{equation}
As is usual, we study the equation for $u$ and later show that we can give meaning to the equation for $\nabla p$, see Lemma \ref{lem:PressureRecovery} and Remark \ref{remark:PressureRecovery}. To this end, let us assume that $z^k$ is smooth and iterate the equation for $u$ into itself to obtain
\begin{align}
\delta u_{st} +\int_s^t P[ (u_r \cdot \nabla)  u_r]\,dr&=\int_s^t \vartheta\Delta u_rdr+ [\clA_{st}^{1}+\clA_{st}^{2}]u_s +u_{st}^{\natural}  \label{eq:RNSRoughFormU}, 
\end{align}
where we have defined  
\begin{align}
\clA^{1}_{st}\phi := P  \tilde{\mathcal{L}}_{\sigma_k} \phi   \, Z_{st}^{k} , \quad 
\clA^{2}_{st}\phi:=P  \tilde{\mathcal{L}}_{\sigma_k} P \tilde{\mathcal{L}}_{\sigma_l}\phi  \mathbb{Z}^{l,k}_{st}, \quad \tilde{\mathcal{L}}_{\sigma_k}\phi:=(\sigma_k\cdot \nabla)\phi +(\nabla \sigma_k)\phi,\label{eq:URDdef} 
\end{align}
\begin{align}
u_{st}^{\natural} &:= \int_s^t P   \tilde{\mathcal{L}}_{\sigma_k} \delta \mu_{sr}  \, dz^k_r  +  \int_s^t P  \tilde{\mathcal{L}}_{\sigma_k}\int_s^r P  \tilde{\mathcal{L}}_{\sigma_l} \left(  \delta \mu_{sr_1}  + P \int_s^{r_1}  \tilde{\mathcal{L}}_{\sigma_m} u_{r_2} \, dz^m_{r_2}   \right) \, dz^l_{r_1} \, dz^k_r, \label{eq:ExplicitRemainder} 
\end{align}
and
\begin{equation} \label{eq:DriftInH0}
\mu_{t}=\int_0^{t} \left[\vartheta\Delta u_r - P (u_r \cdot \nabla) u_r\right]  dr .
\end{equation}
In the above, we have used  suggestive notation for the operator $\tilde{\mathcal{L}}_{\sigma_k}$ as a reminder that it related to the Lie derivative operator. Since $u$ is a vector field, the  Lie-derivative of $u$ by $\sigma_k$ is given by $\clL_{\sigma_k}u:=(\sigma_k\cdot \nabla)\phi -(u\cdot \nabla)\sigma_k $, which is clearly not that same as $\tilde{\clL}_{\sigma_k}u$ unless $\sigma_k$ is constant in space. However, the Lie derivative of the one-form $u^{\flat}=u^idx^i$ associated with $u=u^i\partial_{x^i}$ by $\sigma_k$ is given by $\clL_{\sigma_k}u^{\flat}=(\sigma_k\cdot \nabla)u^{\flat} +(\nabla \sigma_k)u^{\flat}$. Thus,  $\tilde{\clL}_{\sigma_k}u=(\clL_{\sigma_k}u^{\flat})^{\sharp}=(\sigma_k\cdot \nabla)u +(\nabla \sigma_k)u$. 

Since we expect $\mu \in C^{1 -\textnormal{var}}([0,T]; \bH^{0})$ and $u \in L_T^{\infty} \bH^1$, the remainder $u^{\natural}$ in \eqref{eq:ExplicitRemainder} is expected to belong to $ C^{\zeta -\textnormal{var}}_2([0,T]; \bH^{-2})$, for some $\zeta <1$. Assume now that $z^k$ is not a smooth path, but we know how to make sense of $\mathbb{Z}$. Then, the only term that lacks a priori meaning in \eqref{eq:RNSRoughFormU} is the term $u^{\natural}$. However, from formal power counting of the integrals in \eqref{eq:ExplicitRemainder} we still expect this term to be a negligible remainder.
Thus, equation \eqref{eq:RNSRoughFormU} is to be understood in the sense that we \emph{define} the remainder term $u^{\natural}$ from the solution $u$. This will be made precise in Definition \ref{def:RNSStrongSol} below. 

The pair $\boldsymbol{\clA}=(\clA^{1}, \clA^{2})$ is an unbounded $p$-rough driver in the sense of  Definition \ref{def:urd} on the scale $(\bH^{n})_{n}$. 
Indeed,  the existence of a control  $\omega_{\clA}$ such that \eqref{ineq:UBRcontrolestimates} holds follows from the discussion in Section \ref{ss:notation} and the fact that $(Z,\mathbb{Z})$ is a  $p$-rough path in the sense of Definition \ref{defi-rough-path}, which also implies  Chen's relation \eqref{eq:chen-relation}.  We note that control  $\omega_{\clA}$  can be chosen to satisfy 
\begin{equation}\label{ineq:Leraycontrolestimate}
\omega_{\clA}(s,t) \le C\omega_{Z}(s,t),\;\;\forall (s,t)\in \Delta_T,
\end{equation}
for a constant $C>0$  depending only on $d$ and the bounds on $\sigma_i$ in $\W^{2, \infty}$.

We will now give our first definition of a solution to \eqref{eq:RNSDiffForm}.
\begin{definition}\label{def:RNSStrongSol}
We say that $u$ is a strong solution of \eqref{eq:RNSDiffForm} up to time $T^*$ if $u : [0,T^*] \rightarrow \bH^1$ is weakly continuous, $u\in L^2_{T^*}\bH^2 \cap L^{\infty}_{T^*} \bH^1$  and $u^{\natural} : \Delta_{T^*} \rightarrow \bH^{-2}$ defined by  
\begin{align} 
u_{st}^{\natural}(\phi)& :=   \delta u_{st} (\phi ) +  \int_s^t  \left[ - \vartheta (\Delta u_r, \phi) + B_P(u_r)(\phi) \right]\,dr   -  u_s([\clA_{st}^{1,*} +\clA_{st}^{2,*}]\phi) , \label{SystemSolutionU} 
\end{align}
for all $ \phi\in \bH^{2}$ satisfies
$ 
u^{\natural} \in C^{\frac{p}{3}-\textnormal{var}}_{2, \varpi,L}([0,T^*]; \bH^{-2}) 
$
for some control $\varpi$ and $L> 0$.
\end{definition}

\begin{remark}
By the regularity assumption on $u$, it follows that the $dr$-integral in  \eqref{SystemSolutionU} is well-defined. \end{remark}

\begin{remark}
It is possible to formulate a rough version of \eqref{eq:RNSDiffForm} without projecting the equation onto the scale of divergence-free spaces by keeping the pressure in the equation. This formulation was discussed in \cite{HLN} for the case of an energy conservative noise, but the computations carry over mutatis mutandis to the current case giving an equivalent formulation of \eqref{eq:RNSRoughFormU} and \eqref{eq:RNSRoughFormPi} below.
\end{remark}

\begin{remark} \label{rem:URDtoIntegral}
The expansion $\clA^1_{st} u_s + \clA^2_{st} u_s + u_{st}^{\natural}$ should be thought of as the projection of the rough path integral; that is, 
$$
\clA^1_{st} u_s + \clA^2_{st} u_s + u_{st}^{\natural}=P \int_s^t [ (\sigma_k \cdot \nabla) + ( \nabla \sigma_k ) ] u_r \,dZ_r^k . 
$$
Indeed, the expression $\clA^1_{st} u_s + \clA^2_{st} u_s$ represents a local expansion of the rough path integral. If $u^{\natural}$ is a remainder, then by the sewing lemma,  \cite[Lemma 4.2]{FrHa14}, this uniquely defines a path representing the rough integral. 
\end{remark}

\subsection{Vorticity formulation}
Applying  the curl operator $\nabla \times\cdot$ to both sides of \eqref{eq:RNSDiffForm}, we obtain 
\begin{equation}\label{eq:vv}
\partial_t \xi + (u \cdot \nabla)  \xi - 1_{d=3}( \xi\cdot \nabla )u  = \vartheta \Delta \xi + \left[ (\sigma_k \cdot \nabla) \xi  - \mathbf{1}_{d=3} ( \xi\cdot \nabla) \sigma_k\right] \dot{z}_t^k.
\end{equation}

Let us suppose that there exists a strong solution $u$ of \eqref{eq:RNSDiffForm} on $[0,T]$ as defined by Definition \ref{def:RNSStrongSol}. To find a rough path version of the vorticity formulation we apply the curl operator 
 $\xi=\nabla \times $ to both sides of  \eqref{eq:RNSRoughFormU}.
Using  properties of the curl operator and
that $\xi=\nabla \times u$ is a weakly continuous function $\xi : [0,T] \rightarrow \dot{\bH}^0 $ with  $\xi\in L^2_T\dot{\bH}^{1}\cap L^{\infty}_T\dot{\bH}^0$, we find  that $\xi^{\natural} : \Delta_T \rightarrow \dot{\bH}^{-3}$ defined  for all $ \phi\in \dot{\bH}^3$ and $(s,t)\in \Delta_T$  by
\begin{equation}\label{eq:vorticityroughformulation}
\xi_{st}^{\natural}(\phi)=  \delta \xi_{st}(\phi)  +  \int_s^t \left[\vartheta (\nabla  \xi_r,\nabla \phi) + (u_r\cdot \nabla )\xi_r)(\phi)-\mathbf{1}_{d=3}( ( u_r\cdot \nabla ) \xi_r )(\phi)\right]\,dr  - \xi_s([A^{1,*}_{st}+A^{2,*}_{st}]\phi).
\end{equation}
satisfies $\xi^{\natural} \in C^{\frac{p}{3}-\textnormal{var}}_{2, \varpi,L}([0,T]; \dot\bH^{-3})$  for some control $\varpi$ and $L> 0$,
where
\begin{equation} \label{eq:CurlOfURD1}
A_{st}^1\phi:=\nabla \times \clA_{st}^{1} \phi =\clL_{\sigma_k}\phi Z^k_{st}= \big((\sigma_k\cdot \nabla)\phi  - \mathbf{1}_{d=3} ( \phi\cdot \nabla ) \sigma_k \big)    Z^k_{st}
\end{equation}
\begin{align} 
A_{st}^2 \phi :=\nabla \times \clA_{st}^{2} \phi &=\clL_{\sigma_k}\clL_{\sigma_l}\phi \mathbb{Z}^{l,k}_{st}=\big( (\sigma_k\cdot \nabla)(\sigma_l\cdot \nabla)\phi-\mathbf{1}_{d=3}(\sigma_k\cdot \nabla) ((\phi \cdot \nabla) \sigma_l ) \notag \\
&\qquad  -\mathbf{1}_{d=3}((\sigma_l\cdot \nabla) \phi \cdot \nabla)\sigma_k +\mathbf{1}_{d=3}  (((\phi\cdot \nabla) \sigma_l)\cdot \nabla) \sigma_k \big)\mathbb{Z}^{l,k}_{st}. \label{eq:CurlOfURD2}
\end{align}
Indeed, the equalities \eqref{eq:CurlOfURD1} and \eqref{eq:CurlOfURD2} follows from the fact that $\nabla \times P = \nabla \times$ and $\nabla \times \tilde{\mathcal{L}}_{\sigma_k} = \mathcal{L}_{\sigma_k} \nabla \times $ on divergence-free vector fields, which can be checked by direct calculation or by appealing to the differential geometry  (see  page 451 in \cite{taylor2013partial}).

It is clear that   $\bA=(A^1,A^2)$  satisfies  \eqref{ineq:UBRcontrolestimates} for the scale  $(\dot{\bH}^{m})_{m}$ with a control $\omega_{A}$.  The control of  $\omega_{A}$ can be chosen so that
$$
\omega_{A}(s,t)\le C\omega_{Z}(s,t),\;\;\forall (s,t)\in\Delta_T,
$$
for a constant $C>0$  depending only on $d$ and the bounds on $\sigma_i$ in $\W^{2, \infty}$.
Notice that when inverting the curl (i.e., applying the Biot-Savart law $\BS$ to $\xi$), we can only recover the mean-free part of $u$, thus we also need to control the mean of $u$. 
We denote by $\bar{u}$ the spatial mean of $u$; that is, $\bar{u}$ is the $d$-dimensional vector with the $m$'th component given by $\bar{u}^m := (u,e_m)$ where $e_m$ is the $m$'th basis vector of $\bR^d$. Furthermore let $v = u - \bar{u}$ be the mean-free part. The remainder of the mean satisfies
$$
\bar{u}_{st}^{\natural,m} :=u_{st}^{\natural}(e_m)=  (\delta u_{st},e_m) - ([\clA_{st}^{1} +\clA_{st}^{2}] u_s, e_m) = \delta \bar{u}_{st}^m - ([\clA_{st}^{1} +\clA_{st}^{2}] \bar{u}_s, e_m) - ([\clA_{st}^{1} +\clA_{st}^{2}] v_s, e_m) .
$$
We see that $(\clA_{st}^i \bar{u}_s, e_m)  = 0$ since $\bar{u}_s$ is constant in space. 
Moreover, we get
\begin{align*}
(\clA^1_{st} v_s, e_m) & = (  P [ (\sigma_k \cdot \nabla) v_s], e_m  )Z_{st}^k +  (P [(\nabla \sigma_k)  v_s], e_m)Z_{st}^k = (v^l_s, \partial_m \sigma_k^l) Z_{st}^k
\end{align*}
and
\begin{align*}
(\clA^2_{st} v_s, e_m) & = (P( (\sigma_k \cdot \nabla) P [\sigma_j \cdot\nabla v_s]), e_m) \mathbb{Z}_{st}^{j,k} + (P( (\sigma_k \cdot\nabla) P [(\nabla \sigma_j)  v_s]), e_m)\mathbb{Z}_{st}^{j,k} \\ 
& \quad+ (P( (\nabla  \sigma_k)  P [ (\sigma_j \cdot \nabla)  v_s])), e_m)\mathbb{Z}_{st}^{j,k}  + (P( (\nabla \sigma_k)  P [(\nabla \sigma_j)  v_s])), e_m) \mathbb{Z}^{j,k}_{st} \\
& =  (P( (\nabla \sigma_k)  P [ (\sigma_j\cdot \nabla)  v_s])), e_m)\mathbb{Z}_{st}^{j,k}  + (P( (\nabla \sigma_k)  P [(\nabla \sigma_j)  v_s])), e_m) \mathbb{Z}^{j,k}_{st} ,
\end{align*}
where we have used that the vector fields are divergence-free. For the remaining terms, we write
$$
(P( (\nabla \sigma_k)  P [ (\sigma_j \cdot  \nabla)  v_s]), e_m) = (P[ (\sigma_j \cdot \nabla)  v_s], \partial_m \sigma_k) = ( (\sigma_j \cdot \nabla ) v_s, \partial_m \sigma_k) = - (v^l_s, \partial_i (\sigma_j^i \partial_m \sigma_k^l) )
$$
and
\begin{align*}
(P( (\nabla  \sigma_k)  P [\nabla \sigma_j  v_s]), e_m) = (P[\nabla \sigma_j v_s], \partial_m \sigma_k) = ( \nabla \sigma_j v_s, \partial_m \sigma_k) .
\end{align*}
Consequently we get that 
$$
\bar{u}^{\natural}_{st} = \delta \bar{u}_{st} - L_{st}^1(v_s) - L_{st}^2(v_s),
$$
where the $m$'th component of $L_{st}^1(v_s)$ and $L_{st}^2(v_s)$ is given by
$$
(v^l_s, \partial_m \sigma_k^l) Z_{st}^k, \quad \textrm{ and} \quad \left( v^l , \partial_n \sigma_j^l \partial_m \sigma_k^n  -  \sigma_j^n \partial_n \partial_m \sigma_k^l  \right)\mathbb{Z}_{st}^{j,k} 
$$
respectively. Notice that we have
\begin{equation} \label{LBounds}
| L^1_{st}(v_s) | \leq C |v_s|_0 \omega_Z(s,t)^{\frac{1}{p}} \quad \textrm{ and} \quad | L^2_{st}(v_s) | \leq C |v_s|_0 \omega_Z(s,t)^{\frac{2}{p}}
\end{equation}
for a constant $C>0$  depending only on $d$ and the bounds on $\sigma_i$ in $\W^{2, \infty}$.

\begin{definition}\label{def:VortSol}
We say that a pair $(\xi,\bar{u})$ is a weak solution of \eqref{eq:vv} up to time $T^*$ if $(\xi,\bar{u}):[0,T^*]\rightarrow \dot{\bH}^0\times \bR^d$ is weakly continuous, $\xi\in L^2_{T^*}\dot{\bH}^{1}\cap L^{\infty}_{T^*}\dot{\bH}^0$, and  $\bar{u}^{\natural}:\Delta_{T^*}\rightarrow \bR^d$ and $\xi^{\natural} : \Delta_{T^*} \rightarrow \dot{\bH}^{-3}$ defined by 
\begin{align}\label{eq:VortRoughForm}
\bar{u}_{st}^{\natural}&:=  \delta \bar{u}_{st} -[L_{st}^{1}+L_{st}^{2}](\operatorname{\BS}\xi_s)\\
\xi_{st}^{\natural}(\phi)&:=  \delta \xi_{st} (\phi ) +  \int_s^t \left[\vartheta(\nabla \xi_r, \nabla   \phi) + \big( [(\operatorname{\BS} \xi_r+\bar{u}_r ) \cdot \nabla ]\xi_r,\phi \big)-\mathbf{1}_{d=3} \big( [\xi_r\cdot \nabla] \operatorname{\BS} \xi_r,\phi \big)\right]\,dr \\
&\qquad  -  \xi_s ([A_{st}^{1,*} +A_{st}^{2,*}]\phi ), \quad \forall \phi \in \bH^3 
\end{align}
satisfy $\xi^{\natural} \in C^{\frac{p}{3}-\textnormal{var}}_{2, \varpi,L}([0,T^*]; \dot{\bH}^{-3})$  and $\bar{u}^{\natural}\in C^{\frac{p}{3}-\textnormal{var}}_{2, \varpi,L}([0,T^*]; \bR^d)$ for some control $\varpi$ and $L> 0$.
\end{definition}

\begin{lemma} \label{lem:EquivalenceDefinitions}
There is a one-to-one correspondence between the solutions defined in Definitions  \ref{def:RNSStrongSol} and  \ref{def:VortSol}.
\end{lemma}
\begin{proof}

If $u$ is a strong solution of \eqref{eq:RNSDiffForm}, then it is clear from the above  that by defining 
$$
\xi=\nabla \times u, \quad \xi^{\natural}=\nabla \times u^{\natural},\quad \bar{u}^{\natural}= (u^{\natural}(e_m))_{m=1}^d,
$$
we obtain a solution in the sense of Definition \ref{def:VortSol}.  For the reverse direction, define 
$$
v=\operatorname{\BS} \xi, \quad v^{\natural}=\operatorname{\BS} \xi^{\natural}, \quad  u=\bar{u}+v, \quad u^{\natural}=v^{\natural}+\bar{u}^{\natural}.
$$
Since $\BS$ is linear and commutes with derivatives, we have
$$
\delta  v =\operatorname{\BS}\delta \xi,\quad  \Delta v=\operatorname{\BS} \Delta \xi, \quad (\bar{u}\cdot \nabla )v=\operatorname{\BS} [(\bar{u}\cdot \nabla )\xi].
$$
Notice that 
$$
\operatorname{curl}(P[ (v\cdot \nabla) v])=v\cdot \nabla\xi-\mathbf{1}_{d=3} (\nabla v) \xi
$$
and
$$
\operatorname{curl} \big([\mathcal{A}^{1}_{st}+\mathcal{A}^{2}_{st}] u-[L_{st}^1+L_{st}^2](v) \big) =[A^{1}_{st}+A^{2}_{st}]\xi,
$$
where both arguments of the curl are divergence and mean-free.
Thus, since $\operatorname{\BS}\circ \operatorname{curl}$ is the identity on the space of divergence and mean-free test functions, we have
$$
P[(v\cdot \nabla) v] =\operatorname{\BS} [(v\cdot\nabla) \xi-\mathbf{1}_{d=3}( \nabla v) \xi]
$$
$$
[\mathcal{A}^{1}_{st}+\mathcal{A}^{2}_{st}] u-[L_{st}^1+L_{st}^2](v)=\operatorname{\BS} ([A^{1}_{st}+A^{2}_{st}]\xi).
$$
Therefore, applying $\operatorname{\BS}$ to the vorticity equation, we get for all  $\phi \in \bH^3$ and $(s,t)\in \Delta_T$ 
\begin{align}
v_{st}^{\natural}(\phi)&:=  \delta v_{st} (\phi ) +  \int_s^t \left[\vartheta(\nabla v_r, \nabla   \phi) + B_P(v_r+\bar{u}_r,v_r)(\phi)\right]\,dr \\
&\qquad   -  u_s ([\mathcal{A}_{st}^{1,*} +\mathcal{A}_{st}^{2,*}]\phi )+[L_{st}^1+L_{st}^2](v_s)(\phi),
\end{align}
and hence for all  $\phi \in \bW^{3,2}$ and $(s,t)\in \Delta_T$ 
\begin{align}
u_{st}^{\natural}(\phi)&=  \delta u_{st} (\phi ) +  \int_s^t \left[\vartheta(\nabla u_r, \nabla   \phi) + B_P(u_r,u_r)(\phi)\right]\,dr -  u_s ([\clA_{st}^{1,*} +\clA_{st}^{2,*}]\phi ).
\end{align}
\end{proof}

\subsection{Main results} \label{sec:MainResults}

Our main results concern existence and uniqueness of solutions, stability with respect to the given data including  the driving signal, and the existence of a random dynamical system generated by the solution. Let us begin with the precise formulation of  the  existence result.

\subsubsection{Existence and uniqueness}
\begin{theorem} \label{existenceThmNoGalerkin}
Let $d\in \{2,3\}.$  Assume that  $\sigma_k \in \W^{2, \infty}_{\Div}$ for each $k\in \{1,\ldots,K\}$. For a given $u_0\in \bH^1$ and $\mathbf{Z}\in  \mathcal{C}^{p-\textnormal{var}}_g([0,T];\bR^K)$,  there exists a time $T^*$ depending only on $\omega_Z$, $|\sigma|_{2, \infty}$ and $|u_0|_1$, and a strong solution up to time $T^*$ of \eqref{eq:RNSDiffForm} satisfying the energy inequality
\begin{equation} \label{VelocityEnergy}
\sup_{ t \in [0,T^*]} |u_t|^2_1  +  \int_0^{T^*} |\nabla^2 u_r|_0^2\,dr \le F(|u_0|_1),
\end{equation}
for a continuous function $F : \bR_+\rightarrow \bR_+$. Moreover, $u\in C^{p-\textnormal{var}}([0,T^*];\bH^{0})$. When $d = 2$, the final time $T^*$ can be taken to be $T$. 
\end{theorem}
The proof of existence of a solution is a consequence of the stronger statement in Theorem \ref{existenceThm} presented in Section \ref{Section:Galerkin}. The theorem is based on a suitable Galerkin approximation, combined with an approximation of the driving signal $z$ by smooth paths. The fact that $u \in C^{p-\textnormal{var}}([0,T^*];\bH^{0})$ follows from the a priori result in Lemma \ref{AprioriVariation}. 

The next result shows how to construct the pressure from the velocity field. The proof of this statement can be found in  Section \ref{sec:AprioriEst}, page \pageref{p}.

\begin{lemma} \label{lem:PressureRecovery}
Given a strong solution $u$ of \eqref{eq:RNSDiffForm} up to time $T^*$,  the pressure $\pi := \int_0^{\cdot} \nabla p_r dr$ can be recovered. More precisely, there exists $\pi \in C^{p-\textnormal{var}}([0,T^*];\bH^{-2}_{\perp})$ satisfying 
\begin{align}
\delta \pi_{st}+\int_s^t Q[ (u_r \cdot \nabla ) u_r)]\,dr&= [\clA_{st}^{Q,1}+\clA_{st}^{Q,2}] u_s + u_{st}^{Q, \natural},  \label{eq:RNSRoughFormPi}
\end{align}
where   
\begin{align}
\clA^{Q,1}_{st}\phi :=    Q [(\sigma_k\cdot \nabla +\nabla \sigma_k )\phi]  \, Z_{st}^{k}, \quad
\clA^{Q,2}_{st}\phi :=   Q[(\sigma_k\cdot \nabla +\nabla \sigma_k  )P[(\sigma_l\cdot \nabla +\nabla \sigma_l  )\phi]] \mathbb{Z}^{l,k}_{st} \notag,
\end{align}
and $u^{Q,\natural} \in C^{\frac{3}{p} - var}_{2, \varpi,L}([0,T^*] ; \bH_{\perp}^{-2})$ and $\pi \in C^{p-\textnormal{var}}([0,T];\bH^{-2}_{\perp})$.
\end{lemma}

\begin{remark} \label{remark:PressureRecovery}
In the lemma above, as in Remark \ref{rem:URDtoIntegral}, we note that $[\clA_{st}^{Q,1}+\clA_{st}^{Q,2}] u_s + u_{st}^{Q, \natural}$ should be thought of as the rough integral $ Q \int_s^t (\sigma_k\cdot \nabla +\nabla \sigma_k ) u_r \,d z_r^k$. Thus, adding $u$ and $\pi$ and using that $P + Q = I$ gives that
$$
\delta u_{st} + \delta \pi_{st} = \int_s^t [ \vartheta \Delta u_r - (u_r \cdot \nabla)u_r ] dr + \int_s^t (\sigma_k\cdot \nabla +\nabla \sigma_k ) u_r \,d z_r^k.
$$
We also remark that the pair $(\clA^{Q,1}, \clA^{Q,2})$ is, in general, \emph{not} an unbounded rough driver on the scale $(\bH_{\perp}^{n})_n$, because it fails to satisfy Chen's relation \eqref{eq:chen-relation}. Nevertheless, we have
\begin{equation} \label{quasiChen}
\delta \clA_{s \theta t}^{Q,2} = \clA_{\theta t}^{Q,1} \clA_{s \theta}^{1},\quad\text{for all}\quad  (s,\theta,t)\in\Delta^{(2)}_T,
\end{equation}
which  is the correct Chen's relation for the system of equations \eqref{eq:RNSRoughFormU} and \eqref{eq:RNSRoughFormPi} needed to recover the pressure from $u$, see the proof of Lemma \ref{lem:PressureRecovery}.
\end{remark}

In dimension two, we obtain  classical enstrophy balance and uniqueness.

\begin{theorem} \label{thm:UniqandEnEq2d}
In dimension two, for a given $u_0\in \bH^1$ and $\mathbf{Z}\in  \mathcal{C}^{p-\textnormal{var}}_g([0,T];\bR^K)$ and $\sigma_k \in \bW^{3, \infty}_{\Div}$ there is at most one strong solution to \eqref{eq:RNSDiffForm}. Moreover, the velocity  belongs to $ C_T\bH^1$, the following enstrophy balance holds
\begin{equation} \label{eq:EnergyEquality}
\sup_{ t \in [0,T]} |\xi_t|^2_0 +2\vartheta \int_0^T |\nabla \xi_r|_0^2\,dr= |\xi_0|_0^2,
\end{equation}
and the velocity satisfies the energy estimate \eqref{VelocityEnergy} for any $T^*\leq T$.
\end{theorem}

Theorem \ref{thm:UniqandEnEq2d} will follow from Theorem \ref{theorem:EnergyEquality} and Theorem \ref{thm:contractiveTheorem} presented in Section \ref{section:ProofOfTheoremUniqueness}.  

\begin{remark}
Except in the case when $\sigma_j$ is constant in space, there is no reason to believe that one could obtain energy \emph{equality} for the velocity $u$, since the multiplicative term $(\nabla \sigma_k) u \dot{z}^k$ will add energy to the system. 
\end{remark}

\subsubsection{Stability}

Owing to Theorem \ref{existenceThmNoGalerkin} and Theorem \ref{thm:UniqandEnEq2d}, in dimension two, there exists a solution map $\Gamma$  that maps  every   initial condition $u_0\in \bH^1$,  family of divergence-free vector fields  $\sigma_k \in \W^{3, \infty}_{\Div} $, $k\in \{1,\ldots,K\}$, and  continuous geometric $p$-rough path $\bZ=(Z,\mathbb{Z})$
to a unique strong solution $u$  of  \eqref{eq:RNSDiffForm}. Let us denote by $\bH^{1}_{w}$ the space $\bH^1$ equipped with its weak topology.
The following stability result is proved in Section \ref{s:stab}. 

\begin{corollary}\label{cor:stability}
In dimension two, the solution map 
\begin{align*}
\Gamma:\bH^1\times \big(\W^{3, \infty}_{\emph{div}} \big)^K \times \mathcal{C}^{p-\textnormal{var}}_g([0,T];\bR^K)&\to L^2_T\bH^1\cap C_T\bH_w^1 \cap C_T \bH^0 \\
(u_0,\sigma,\bZ)&\mapsto u
\end{align*}
is continuous. 

In particular,  the following Wong-Zakai result holds true. Let $\{B^n\}$ be a piecewise linear interpolation of a Brownian motion $B$, and for each $n$, denote by $u^n $ the unique strong solution of \eqref{eq:RNSDiffForm} with $\dot{z}$ is replaced by $\dot{B}^n$, existence of which is guaranteed by Theorem \ref{thm:UniqandEnEq2d} and Theorem \ref{existenceThmNoGalerkin}. Then $\{u^n\}$ converges almost surely to $u$ in $L^2_T\bH^1\cap C_T\bH_w^1 \cap C_T \bH^0$ where $u$ is the strong probabilistic, pathwise unique solution of 
\begin{align*}
du_t + \nabla p_t dt  & = [ \vartheta \Delta u_t - (u_t \cdot \nabla) u_t ] dt + [( \sigma_k \cdot \nabla) + (\nabla \sigma_k )]  u_t \circ dB^k_t \\
\Div u_t & = 0, \quad u_t|_{t = 0} = u_0 \in \bH^1 .
\end{align*}
constructed in \cite[Theorem 2.1]{mikulevicius2005global} for the more general case of $u_0 \in \bH^0$. 

In particular, the energy estimate \eqref{VelocityEnergy} is satisfied for solutions corresponding to almost all sample paths of the Brownian motion. 
\end{corollary}

\begin{remark}
By applying the curl operator to $u$, we also obtain continuity of the mapping
\begin{align*}
\Gamma:\bH^1\times \big(\W^{3, \infty}_{\textrm{div}} \big)^K \times \mathcal{C}^{p-\textnormal{var}}_g([0,T];\bR^K)&\to L^2_T\bH^0\cap C_T\bH_w^0 \cap C_T \bH^{-1}\\
(u_0,\sigma,\bZ)&\mapsto\xi .
\end{align*}
\end{remark}

\subsubsection{Random dynamical system}

Based on our well-posedness and stability result in dimension two, under suitable assumptions on the driving rough path, we are able to construct a continuous random dynamical system corresponding to the Navier-Stokes equations
\eqref{eq:RNSDiffForm}.

Let us first introduce the necessary definitions. Let $(\Omega, \mathcal{F})$
and $(X, \mathcal{B})$ be measurable spaces. A family $\theta =
(\theta_t)_{t\geq 0}$ of maps from $\Omega$ to itself is a measurable
dynamical system provided
\begin{enumerate}
  \item $(t, \omega) \mapsto \theta_t \omega$ is $\mathcal{B} ([0,\infty))
  \otimes \mathcal{F} / \mathcal{F}$-measurable,
  
  \item $\theta_0 = \tmop{Id}$,
  
  \item $\theta_{s + t} = \theta_t \circ \theta_s$ for all $s, t \geq 0$.
\end{enumerate}
If $\mathbb{P}$ is a probability measure on $(\Omega, \mathcal{F})$ that is
invariant under $\theta$, i.e. $\mathbb{P} \circ \theta^{- 1}_t
=\mathbb{P}$ for all $t \geq 0$, we call $(\Omega, \mathcal{F},
\mathbb{P}, \theta)$ a measurable metric dynamical system.  A measurable random dynamical system on
$(X, \mathcal{B})$ is a measurable metric dynamical system $(\Omega,
\mathcal{F}, \mathbb{P}, \theta)$ together with a measurable map $\varphi :
[0,\infty) \times \Omega \times X \rightarrow X$ satisfying the cocycle
property: that is, $\varphi_0 (\omega) = \tmop{Id}_X$ for all $\omega \in \Omega$,
and
\[ \varphi_{s + t} (\omega) = \varphi_t (\theta_s \omega) \circ \varphi_s
   (\omega) \]
for all $s, t \in [0,\infty)$ and $\omega \in \Omega$. If, in addition, $X$ is a topological
space and the map $\varphi (\cdummy, \omega, \cdummy) : [0,\infty)\times X
\rightarrow X$ is continuous for every $\omega \in \Omega$, it is called a
continuous random dynamical system.

Under suitable assumptions on the coefficients,
rough path driven differential equations generate random dynamical systems
provided the driving rough path is a rough path cocycle  \cite{BAILLEUL20175792}. To be more precise,
if $p \in [2, 3)$ and $(\Omega, \mathcal{F}, \mathbb{P}, \theta)$ is a
measurable metric dynamical system, then we say that
\[ \bZ = (Z, \mathbb{Z}) : \Omega \rightarrow C_{\rm loc}^{p - \tmop{var}} ([0,
   \infty); \bR^K) \times C_{2,\rm loc}^{p - \tmop{var}} ([0, \infty) ; \bR^{K
   \times K}) \]
is a continuous $p$-rough path cocycle provided $\bZ (\omega)|_{[0,T]}$ is a
continuous $p$-rough path for every $T>0$ and $\omega \in \Omega$ and the following
cocycle property
\[ Z_{s, s + t} (\omega) = Z_t (\theta_s \omega), \qquad \mathbb{Z}_{s, s + t}
   (\omega) =\mathbb{Z}_{0, t} (\theta_s \omega), \]
holds true for every $s, t \geq 0$ and $\omega \in \Omega$. Similarly,
one may  define a $p$-rough path cocycle for any $p \in [1, \infty)$.
It was shown in Section 2 in \cite{BAILLEUL20175792} that rough path lifts of various stochastic
processes define cocycles. These  include Gaussian processes with
stationary increments and independent components under certain assumption on
the covariance, satisfied for instance by the fractional Brownian motion with
Hurst parameter $H > 1 / 4$.

\begin{corollary}
Assume that the driving rough path $\bZ$ is a continuous $p$-rough
path cocycle for some $p \in [2, 3)$ and let $d=2$. Then the system \eqref{eq:RNSDiffForm} generates a continuous random dynamical system on $\bH^{1}$.
\end{corollary}

\begin{proof}
Let $\psi : [0, \infty) \times [0,
\infty) \times \Omega \times \bH^1 \rightarrow \bH^1$ be the
random flow generated by \eqref{eq:RNSDiffForm}, that is, $\psi (t, s, \omega, u_0)$ is the
unique solution to \eqref{eq:RNSDiffForm} starting at time $s$ from the initial condition
$u_0$, driven by $\bZ (\omega)$ and evaluated at the time $t$. Then
it follows from our definition of solution and the cocycle property of
$\bZ$ that $\psi (t + h, s + h, \omega, u_0) = \psi (t, s, \theta_h
\omega, u_0)$. Consequently, we define $\varphi (t, \omega, u_0) = \psi (t, 0,
\omega, u_0)$ and using also the semiflow property of $\psi$ which follows
from uniqueness, we deduce that $\varphi$ has the cocycle property. The continuity with respect to time as well as the initial condition follows from \eqref{eq:Contraction}.
\end{proof}

\color{black}

\section{A priori estimates}\label{sec:AprioriEst}

In this section, we derive a priori estimates of the remainder term $u^{\natural}$ and  $|u|_{p-\textnormal{var};[0,T];\bH^{0}}$. Although similar estimates have been derived in \cite{HLN}, there are subtle differences in the present paper motivating us to include them also here. Namely, since we are dealing with strong solutions, all the estimates have to be done with one extra derivative, but this is more than just shifting the scale we are working on. Indeed, technical computations involving the non-linearity and the remainder term show that it is no longer necessary  to introduce a scale of fractional Sobolev spaces that were needed in \cite{HLN}.

Let $u$ be a  solution of \eqref{eq:RNSDiffForm} in the sense of Definition \ref{def:RNSStrongSol}. We recall the definition of $\mu$ in \eqref{eq:DriftInH0} which we restrict to the scale $(\bH^n)_{n }$. 
It follows that for  $(s,t)\in \Delta_T$,
\begin{equation}\label{RNS2}
\delta u_{st} =  \delta \mu_{st} + \clA_{st}^{1} u_s+ \clA_{st}^{2}u_s + u^{\natural}_{st},
\end{equation}
where the equality holds in $\bH^{-2}$.
Using  \eqref{trilinear form estimate} with $m_1, m_3 = 1$ and $m_2 = 0$ we obtain $|B_P(u_r)|_{-1} \lesssim |u_r|_1^2$, and hence
\begin{equation} \label{DriftLipschitz}
|\delta \mu_{st} |_{-1} \lesssim (1 + |u|_{L^{\infty} \bH^1})^2 (t-s).
\end{equation}

Let us  define the intermediate remainder 
\begin{equation}\label{eq:defin of sharp}
u^\sharp_{st}:=\delta u_{s t}-\clA^{1}_{st}u_s=\delta\mu_{st}+\clA^{2}_{st}u_s+u^{\natural}_{st}.
\end{equation}
Notice that the first expression has low regularity in time but is not very irregular in space, whereas the second one has higher time regularity but less regularity in space. Interpolating these expressions with the help of  the smoothing operators from Definition \ref{def:smoothingOp} allows us to obtain a priori estimates on the remainder $u^{\natural}$. The following result is similar as in \cite[Lemma 3.1]{HLN}, but with the added spatial regularity of the solution, which allows us to circumvent the use of fractional Sobolev spaces as in \cite{HLN}.

\begin{lemma} \label{Thm2.5}
Assume that $u$ is a solution of  \eqref{eq:RNSDiffForm} in the sense of Definition \ref{def:RNSStrongSol}. For $(s,t)\in \Delta_T$ such that  $\varpi(s,t)\le L$, let $\omega_{\natural }(s,t) := |u^{\natural}|^{\frac{p}{3}}_{\frac{p}{3} -var; [s,t];\bH^{-2}}.$
Then  there is a constant $\tilde{L}>0$, depending only on $p$ and $d$,  such that  for all $(s,t)\in \Delta_T$ with  $\varpi(s,t)\le L$ and $\omega_{A}(s,t)\leq \tilde{L}$,
\begin{equation} \label{RemainderEstimateWithoutMu}
\omega_{\natural }(s,t)  \lesssim_{p} |u|^{\frac{p}{3}}_{L^{\infty}_T\bH^1}  \omega_{\clA}(s,t) + ( 1 + |u|_{L^{\infty}_T\bH^1} )^{\frac{2p}{3}}(t-s)^{\frac{p}{3}} \omega_{\clA}(s,t)^{\frac13}.
\end{equation}
\end{lemma}
\begin{proof}
Applying $\delta$ to \eqref{SystemSolutionU}, we find  that for all $\phi \in \bH^{2}$ and $(s,\theta,t)\in \Delta_T,$
$$
\delta u_{s \theta t}^{\natural }( \phi) = \delta u_{s \theta} (\clA_{\theta t}^{2,*} \phi) + u^\sharp_{s\theta}(\clA^{1,*}_{\theta t} \phi),
$$
where $u^\sharp_{s\theta}$ is defined in \eqref{eq:defin of sharp}. 
We decompose $\delta u^{\natural }_{s \theta t} (\phi)$  into a smooth (in space) and non-smooth part using the smoothing operator $J^{\eta}$ to get 
$$
\delta u^{\natural }_{s \theta t} (\phi)  = (I - J^{\eta})\delta u^{\natural }_{s \theta t} (\phi) + J^{\eta}\delta u^{\natural }_{s \theta t} ( \phi) ,
$$ 
for some $\eta \in (0,1]$ that will be specified later.
To estimate the smooth part, we use \eqref{smoothingOperator} and that $u^{\sharp}_{s\theta}= \delta u_{s \theta}-\clA^{1}_{s\theta}u_s$ to obtain
\begin{align*}
\left|(I - J^{\eta} )\delta u^{\natural }_{s \theta t} (\phi)\right| & \le    |u|_{L^{\infty}_T\bH^1}\left( \left|(I-J^{\eta}) \clA_{\theta t}^{1*} (\phi)\right|_{-1}  +  \left|(I-J^{\eta}) \clA_{s \theta}^{1,*} \clA_{\theta t}^{1*} \phi\right|_{-1}  + \left|(I-J^{\eta}) \clA^{2*}_{\theta t} \phi\right|_{-1}\right) \\
& \lesssim |u|_{L^{\infty}_T\bH^1} \left( \eta^2 \left| \clA_{\theta t}^{1*} (\phi)\right|_{1}  +  \eta \left| \clA_{s \theta}^{1,*} \clA_{\theta t}^{1*} \phi\right|_{0}  + \eta \left|\clA^{2*}_{\theta t} \phi\right|_{0}\right)\\
& \lesssim |u|_{L^{\infty}_T\bH^1} \left( \omega_{\clA}(s,t)^{\frac{1}{p}} \eta^2 + \omega_{\clA}(s,t)^{\frac{2}{p}} \eta \right) |\phi|_2.
\end{align*}
In order to estimate the non-smooth part, we use  the form $u^\sharp_{s \theta}=  \delta \mu_{s \theta } + \clA_{s \theta}^{2} u_s + u_{s \theta}^{\natural }$ to get
\begin{align*}
J^{\eta}\delta u_{s \theta t}^{\natural } ( \phi) & =  \delta \mu_{s \theta}  (J^{\eta}\clA^{1,*}_{\theta t} \phi) + u_s(J^{\eta} \clA^{2,*}_{s \theta} \clA_{\theta t}^{1,*} \phi) + u_{s \theta}^{\natural }(J^{\eta}\clA_{\theta t}^{1,*} \phi) \\
& \quad+ \delta \mu_{s \theta}  (J^{\eta}\clA^{2,*}_{\theta t} \phi) + u_s(J^{\eta} \clA^{1,*}_{s \theta} \clA_{\theta t}^{2,*} \phi)+ u_s( J^{\eta}\clA^{2,*}_{s \theta} \clA_{\theta t}^{2,*} \phi) + u_{s \theta}^{\natural }(J^{\eta}\clA_{\theta t}^{2,*} \phi) ,
\end{align*}
Estimating each term and using \eqref{smoothingOperator} and \eqref{DriftLipschitz}, for all $(s,\theta,t)\in \Delta_T^{(2)}$ such that $\varpi(s,t)\le L$,  we find
\begin{align}
|J^{\eta}\delta u^{\natural }_{s \theta t} ( \phi ) | & \lesssim (1 + |u|_{L^{\infty}_T \bH^1})^2 (t-s)  |\clA^{1,*}_{\theta t} \phi|_1 + |u|_{L_T^{\infty}\bH^1}  |\clA^{2,*}_{s \theta} \clA_{\theta t}^{1,*} \phi|_2 + \omega_{\natural }(s,t)^{\frac{3}{p}} \eta^{-1}| \clA_{\theta t}^{1,*} \phi|_2\notag \\
&\quad + (1 + |u|_{L^{\infty}_T \bH^1})^2 (t-s)  \eta^{-1} |\clA^{2,*}_{\theta t} \phi |_0 + |u|_{L_T^{\infty}\bH^1} |\clA^{1,*}_{s \theta} \clA_{\theta t}^{2,*}\phi|_{-1} + |u|_{L_T^{\infty}\bH^1} \eta^{-1} |\clA^{2,*}_{s \theta} \clA_{\theta t}^{2,*}\phi|_{-2} \notag \\
&\quad + \omega_{\natural }(s,t)^{\frac{3}{p}} \eta^{-2} | \clA_{\theta t}^{2,*} \phi|_0 \notag \\
& \lesssim  \left( (1 + |u|_{L^{\infty}_T \bH^1})^2 (t-s) \omega_{\clA}(s,t)^{\frac{1}{p}} + |u|_{L_T^{\infty}\bH^1} \omega_{\clA}(s,t)^{\frac{3}{p}} + \omega_{\natural }(s,t)^{\frac{3}{p}} \omega_{\clA}(s,t)^{\frac{1}{p}} \eta^{-1}  \right.\notag  \\
&\quad \left. + (1 + |u|_{L^{\infty}_T \bH^1})^2 (t-s) \omega_{\clA}(s,t)^{\frac{2}{p}} \eta^{-1} + |u|_{L_T^{\infty}\bH^1} \omega_{\clA}(s,t)^{\frac{3}{p}} + |u|_{L_T^{\infty}\bH^1} \omega_{\clA}(s,t)^{\frac{4}{p}}  \eta^{-1} \right.\notag \\
&\quad +\left.  \omega_{\natural }(s,t)^{\frac{3}{p}} \omega_{\clA}(s,t)^{\frac{2}{p}} \eta^{-2} \right) |\phi|_2  . \label{ineq:remainderestlastline}
\end{align}
Setting $\eta = \omega_{\clA}(s,t)^{\frac{1}{p}} \lambda $ for some constant $\lambda > 0$ to be determined later,  we have 
\begin{align*}
|\delta u^{\natural }_{s \theta t} |_{-2} & \lesssim   |u|_{L^{\infty}_T\bH^1}  \omega_{\clA}(s,t)^{\frac{3}{p}} ( \lambda^{-1} + 1 + \lambda + \lambda^2) + (t-s) \omega_{\clA}(s,t)^{\frac{1}{p}} \\
&\quad + (1 + |u|_{L^{\infty}_T \bH^1})^2 (t-s) \omega_{\clA}(s,t)^{\frac{2}{p}}  +  \omega_{\natural }(s,t)^{\frac{3}{p}} (\lambda^{-1}+ \lambda^{-2}) \\
& \lesssim_{p}  \left( |u|^{\frac{p}{3}}_{L^{\infty}_T\bH^1}  \omega_{\clA}(s,t) ( \lambda^{-1} + 1 + \lambda + \lambda^2)^{\frac{p}{3}} + (t-s)^{\frac{p}{3}} \omega_{\clA}(s,t)^{\frac{1}{3}}  \right.\\
& \quad+  \left. (1 + |u|_{L^{\infty}_T \bH^1})^{\frac{2p}{3}}(t-s)^{\frac{p}{3}} \omega_{\clA}(s,t)^{\frac{2}{3}}  +  \omega_{\natural }(s,t)(\lambda^{-1}+ \lambda^{-2})^{\frac{p}{3}} \right)^{\frac{3}{p}} .
\end{align*}
Applying the sewing lemma, Lemma 2.1 \cite{DeGuHoTi16}, we get
\begin{align*}
| u^{\natural }_{s t} |_{-2}^{\frac{p}{3}}&\lesssim_{p} |u|^{\frac{p}{3}}_{L^{\infty}_T\bH^1}  \omega_{\clA}(s,t) ( \lambda^{-1} + 1 + \lambda + \lambda^2)^{\frac{p}{3}} + (t-s)^{\frac{p}{3}} \omega_{\clA}(s,t)^{\frac{1}{3}}  .\\
&\quad \quad +  (1 + |u|_{L^{\infty}_T \bH^1})^{\frac{2p}{3}} (t-s)^{\frac{p}{3}} \omega_{\clA}(s,t)^{\frac{2}{3}}  +  \omega_{\natural }(s,t)(\lambda^{-1}+ \lambda^{-2})^{\frac{p}{3}}.
\end{align*}
Since $\omega_{\natural}=|u^{\natural}|^{\frac{p}{3}}_{\frac{p}{3} -var; [s,t];\bH^{-2}}$ is equal to the infimum over all controls satisfying $|u_{st}^{\natural }|_{-2} \leq \omega_{\natural }(s,t)^{\frac{3}{p}}$ (see  \eqref{smoothingOperator}), there is a constant $C=C(p,d)$ such that 
\begin{align*}
\omega_{\natural }(s, t) &\le C  \left( |u|^{\frac{p}{3}}_{L^{\infty}_T\bH^1}  \omega_{\clA}(s,t) ( \lambda^{-1} + 1 + \lambda + \lambda^2)^{\frac{p}{3}} + (t-s)^{\frac{p}{3}} \omega_{\clA}(s,t)^{\frac{1}{3}}  \right.\\
&\quad \quad\quad +  \left. (1 + |u|_{L^{\infty}_T \bH^1})^{\frac{2p}{3}} (t-s)^{\frac{p}{3}} \omega_{\clA}(s,t)^{\frac{2}{3}}  +  \omega_{\natural }(s,t)(\lambda^{-1}+ \lambda^{-2})^{\frac{p}{3}} \right).
\end{align*}
Choosing $\lambda$ such that $C(\lambda^{-1}+ \lambda^{-2})^{\frac{p}{3}} \leq \frac{1}{2}$ and $\tilde{L}>0$ such that $\eta=\omega_{\clA}(s,t)^{\frac{1}{p}} \lambda \leq \tilde{L}\lambda\le  1$, we obtain \eqref{RemainderEstimateWithoutMu}.
\end{proof}

We go on to prove an a priori estimate on the $p$-variation of the solution $u$.
\begin{lemma} \label{AprioriVariation}
Assume that $u$ is a solution of \eqref{eq:RNSDiffForm} in the sense of Definition \ref{def:RNSStrongSol}. Then $u\in C^{p-\textnormal{var}}([0,T];\bH^{0})$ and  there is a constant $\tilde{L}>0$, depending only on $p$ and $d$,   such that for all $(s,t)\in \Delta_T$ with $\varpi(s,t)\le L$, $\omega_{A}(s,t)\leq \tilde{L}$, and $\omega_{\natural}(s,t)\leq \tilde{L}$, it holds that
\begin{equation} \label{UVariationWithoutMu}
\omega_u(s,t) \lesssim_{p} (1 + |u|_{L^{\infty}_T\bH^1})^{2p} (\omega_{\natural }(s,t) + (t-s) + \omega_{\clA}(s,t)), 
\end{equation}
where $\omega_u (s,t) := |u|^p_{p - \textnormal{var}; [s,t];\bH^{0}}$.
\end{lemma}
\begin{proof}
For  all $\eta \in (0,1]$, $(s,t)\in \Delta_T$ and $\phi \in \bH^0$, we have
$$
\delta u_{st} (\phi)  = \delta u_{st}(J^{\eta} \phi)+ \delta u_{st}((I - J^{\eta})\phi) .
$$
Applying \eqref{smoothingOperator}, we find
$$
|\delta u_{st}((I - J^{\eta})\phi)| \leq 2 |u|_{L^{\infty}_T\bH^1} |(I - J^{\eta})\phi|_{-1} \lesssim \eta |u|_{L^{\infty}_T\bH^1} |\phi|_0 .
$$
In order to estimate the smooth part, we expand $\delta u_{st}$ using \eqref{RNS2} and then apply \eqref{smoothingOperator} to get
\begin{align*}
|\delta u_{st}(J^{\eta} \phi)| & \leq |u_{st}^{\natural}  (J^{\eta} \phi) | + |\delta \mu_{st} (J^{\eta} \phi)| + |u_s (\clA^{1,*}_{st} J^{\eta} \phi) |+ |u_s (\clA^{2,*}_{st} J^{\eta} \phi) | \\
& \lesssim \omega_{\natural }(s,t)^{\frac{3}{p}} |J^{\eta} \phi |_{2} + (1+|u|_{L^{\infty}_T\bH^1})^2 (t-s) |J^{\eta} \phi|_1 + |u|_{L^{\infty}_T\bH^1} \omega_{\clA}(s,t)^{\frac{1}{p}} | J^{\eta} \phi |_{0} + |u|_{L^{\infty}_T\bH^1} \omega_{\clA}(s,t)^{\frac{2}{p}} | J^{\eta} \phi |_{1} \\
& \lesssim \left( \omega_{\natural }(s,t)^{\frac{3}{p}} \eta^{-2}  + (1+|u|_{L^{\infty}_T\bH^1})^2 (t-s) \eta^{-1}  + |u|_{L^{\infty}_T\bH^1} \omega_{\clA}(s,t)^{\frac{1}{p}}  + |u|_{L^{\infty}_T\bH^1} \omega_{\clA}(s,t)^{\frac{2}{p}} \eta^{-1}  \right)  |\phi |_{0} ,
\end{align*}
for all $(s,t)\in \Delta_T$ such that $\varpi(s,t)\le L$.
Setting $\eta = \omega_{\natural }(s,t)^{\frac{1}{p}} + \omega_{\clA}(s,t)^{\frac{1}{p}} + (t-s)^{\frac{1}{p}}$ and choosing $\tilde{L}>0$ such that $\eta \in (0,1]$, we get
\begin{align*}
|\delta u_{st} |_{0} & \lesssim_{p} (1+  |u|_{L^{\infty}_T\bH^1})^2 \left(  \omega_{\natural }(s,t) + (t-s) +   \omega_{\clA}(s,t)\right)^{\frac{1}{p}},
\end{align*}
which proves the claim.
\end{proof}

The following lemma  shows that the solution $u$ is controlled by $\clA^{1}$ so that we may construct the rough integral $Q \int_0^{\cdot}  \big(\sigma_k \cdot\nabla  + (\nabla \sigma_k) \big) u_r \,dz_r^k$ needed to recover the pressure.

\begin{lemma} \label{AprioriVariation1}
Assume that $u$ is a solution of \eqref{eq:RNSDiffForm} in the sense of Definition \ref{def:RNSStrongSol}. Then there is a constant $\tilde{L}>0$, depending only on $p$ and $d$,   such that for all $(s,t)\in \Delta_T$ with $\varpi(s,t)\le L$, $\omega_{A}(s,t)\leq \tilde{L}$, and $\omega_{\natural}(s,t)\leq \tilde{L}$, it holds that
$$
\omega_\sharp(s,t) \lesssim_{p}  (1 + |u|_{L^{\infty}_T\bH^1})^{\frac{2}{p}} (\omega_{\natural }(s,t) + (t-s)+ \omega_{\clA}(s,t)), 
$$
where $\omega_\sharp (s,t) := |u^\sharp|^\frac{p}{2}_{\frac{p}{2} - \textnormal{var}; [s,t];\bH^{-1}}$.
\end{lemma}
\begin{proof}
For all  $\eta \in (0,1]$, $(s,t)\in \Delta_T$ and $\phi \in \bH^1$, we have
$$
u^\sharp_{st} (\phi)  = u^\sharp_{st}(J^{\eta} \phi)+ u^\sharp_{st}((I - J^{\eta})\phi) .
$$
We recall \eqref{eq:defin of sharp}, giving two expressions for $u^\sharp$.
As explained above, we employ the first formula to estimate the non-smooth part and the second one to estimate  the smooth part. Applying \eqref{smoothingOperator}, we find
\begin{align*}
|u^\sharp_{st}((I - J^{\eta})\phi)|&\leq|\delta u_{s t}((I - J^{\eta})\phi)|+|u_s(\clA^{1,*}_{st}(I - J^{\eta})\phi)| \\
&\le |u|_{L^{\infty}_T\bH^1}|(I - J^{\eta})\phi|_1 +|u|_{L^{\infty}_T\bH^1}\omega_{\clA}(s,t)^\frac{1}{p} |(I - J^{\eta})\phi|_0 \\
&\lesssim \Big(\eta^2 |u|_{L^{\infty}_T\bH^1} +\eta |u|_{L^{\infty}_T\bH^1}\omega_{\clA}(s,t)^\frac{1}{p} \Big)|\phi|_1 .
\end{align*}
In order to estimate  the non-smooth part, we apply  \eqref{smoothingOperator} to obtain
\begin{align*}
|u^\sharp_{st}(J^{\eta} \phi)| & \leq |u_{st}^{\natural}  (J^{\eta} \phi) | + |\delta \mu_{st} (J^{\eta} \phi)| +  |u_s (\clA^{2,*}_{st} J^{\eta} \phi) | \\
& \lesssim\omega_{\natural }(s,t)^{\frac{3}{p}} |J^{\eta} \phi |_{2} + (1 + |u|_{L^{\infty}_T \bH^1})^2(t-s) |J^{\eta} \phi|_1 +  |u|_{L^{\infty}_T\bH^1} \omega_{\clA}(s,t)^{\frac{2}{p}} | J^{\eta} \phi |_{1} \\
& \leq \left( \omega_{\natural }(s,t)^{\frac{3}{p}} \eta^{-1}  + (1 + |u|_{L^{\infty}_T \bH^1})^2(t-s) \eta^{-1}  + |u|_{L^{\infty}_T\bH^1} \omega_{\clA}(s,t)^{\frac{2}{p}}  \right)  |\phi |_{1} ,
\end{align*}
for all $(s,t)\in \Delta_T$ with $\varpi(s,t)\le L$.
Setting $\eta = \omega_{\natural }(s,t)^{\frac{1}{p}} + \omega_{\clA}(s,t)^{\frac{1}{p}} + (t-s)^{\frac{1}{p}}$ and choosing $\tilde{L}>0$ such that $\eta \in (0,1]$, we find
\begin{align*}
| u^\sharp_{st} |_{-1}  \lesssim_{p} (1+  |u|_{L^{\infty}_T\bH^1})^2 \left(  \omega_{\natural }(s,t) + (t-s) + \omega_{\clA}(s,t)\right)^{\frac{2}{p}},
\end{align*}
which proves the claim.
\end{proof}

Finally, we have all in hand to show how to recover the pressure in the original equation \eqref{RNS2}. The computation in the proof shows why \eqref{quasiChen} is the correct Chen's relation for this system.

\label{p}
\begin{proof}[Proof of Lemma \ref{lem:PressureRecovery}]

We first show that we can construct the rough integral $$I_t=Q\int_0^t \big[ \sigma_k \cdot\nabla + (\nabla \sigma_k) \big] u_r \, d z^k_r, \quad I_0=0,$$ using the sewing lemma, \cite[Lemma 2.1]{DeGuHoTi16}. Let
$
h_{st} = \clA^{Q,1}_{st}u_s+ \clA^{Q,2}_{st}u_s
$
for $(s,t)\in \Delta_T$. It follows that   $h\in C_2^{p-\textnormal{var}}([0,T];\bH_{\perp}^{-1})$.
Straightforward computations show that, using \eqref{quasiChen}, we have $\delta h_{s\theta t} = -\clA^{Q,1}_{\theta t}u^\sharp_{s\theta}-\clA^{Q,2}_{\theta t} \delta u_{s\theta}$.
Owing to Lemmas \ref{AprioriVariation} and  \ref{AprioriVariation1} there are controls $\omega$  and $\varpi$ and an $L>0$  such that for all $(s,\theta,t)$ with $\varpi(s,t)\le L$, we have
$$
|\delta h_{s\theta t}|_{-2}\lesssim_p  \left(\omega_{\clA}(s,t)^{\frac{1}{3}}\omega_{\sharp}(s,t)^{\frac{2}{3}}+\omega_{\clA}(s,t)^{\frac{2}{3}}\omega_{u}(s,t)^{\frac{1}{3}}\right)^{\frac{3}{p}}=:\omega(s,t)^{\frac{3}{p}}.
$$
Therefore, by the sewing lemma, Lemma 2.1 \cite{DeGuHoTi16},  there exists a unique path $I\in C^{p-\textnormal{var}}([0,T];\bH^{-2}_{\perp})$ and a two-index map $ I^{\natural}\in  C_{2,\varpi,L}^{p-\textnormal{var}}([0,T];\bH_{\perp}^{-2})$ such that 
$$
\delta I_{st}= \clA^{Q,1}_{st}u_s+ \clA^{Q,2}_{st}u_s + I^\natural_{st} . 
$$
and
$| I^\natural_{s t}|_{-2}\lesssim_p \omega(s,t)^\frac{3}{p}.$  

Defining
$
\pi_t :=  -  \int_0^t B_Q(u_r) \, dr   + I_t ,
$
gives exactly \eqref{eq:RNSRoughFormPi} with $u^{Q, \natural} := I^{\natural}$. 
A direct estimate shows that $\pi \in C^{p - var}([0,T]; \bH_{\perp}^{-2})$.
\end{proof}

When proving existence using a Galerkin approximation, we will use Definition \ref{def:VortSol} to find estimates as indicated by Theorem \ref{thm:UniqandEnEq2d}, since $\xi$ satisfies an enstrophy balance. 
However, using the Biot-Savart law  only yields an estimate on the mean-free part of the velocity $v = u - \bar{u}$.
The next lemma shows how to bound the mean, $\bar{u}$, in terms of $v$.

\begin{lemma} \label{lem:aPrioriMean}
Assume $(\xi,\bar{u})$ is a solution of \eqref{eq:vv} up to time $T$. Then there exists a constant $C$ depending only on $\omega_{Z}$ and $p$ such that
\begin{equation} \label{eq:aPrioriMean}
|\bar{u}|_{L^{\infty}_T \bR^d} \leq C \exp \left\{ C(1 + |v|_{L^{\infty}_T \bH^1})^p \right\} \left(1 +   |\bar{u}_0|  \right).
\end{equation}
\end{lemma}

\begin{proof}
From Theorem \ref{Thm2.5Abstract} we get (notice the decreased spatial regularity)
\begin{align} \label{YetAnotherAprioriEstimate}
|u_{st}^{\natural}|_{-3} &  \lesssim |u|_{L^{\infty}([s,t];\bH^0)} \omega_{\clA}(s,t)^{\frac{3}{p}} +  \omega_{\clA}(s,t)^{\frac{1}{p}} |\mu|_{1-var; [s,t]; \bH^{-1}}
\end{align}
where $\mu_t$ from \eqref{eq:DriftInH0} is regarded as a bounded variation path with values in $\bH^{-1}$. By the bilinearity of $B_P$ and since $\nabla u = \nabla v$ we write
$$
|B_P(u,u)|_{-1} = |B_P(u,v)|_{-1} \leq  |B_P(\bar{u},v)|_{-1} + |B_P(v,v)|_{-1} \lesssim |\bar{u}| |v|_{1} + |v|_{1}^2
$$
where the last inequality comes from setting $m_1 = m_3 = 1$ and $m_2 = 0$ in \eqref{trilinear form estimate}. This gives the bound
$$
\left|\delta\mu_{st} \right|_{-1} \lesssim |\bar{u}|_{L^{\infty}([s,t])} |v|_{L^{\infty}([s,t]; \bH^1)} (t-s) +  |v|_{L^{\infty}([s,t]; \bH^1)}^2 (t-s) .
$$

Moreover, recall the definition of $\bar{u}$, \eqref{eq:VortRoughForm}. From \eqref{YetAnotherAprioriEstimate} and Lemma \ref{lem:EquivalenceDefinitions} we get that 
$$
|\bar{u}_{st}^{\natural}| \leq |u_{st}^{\natural}|_{-3} \lesssim |\bar{u}|_{L^{\infty}([s,t])} (1 + |v|_{L^{\infty}([s,t]; \bH^1)})  \omega_{\clA}(s,t)^{\frac{1}{p}} +  (1+ |v|_{L^{\infty}([s,t]; \bH^1)})^2  \omega_{\clA}(s,t)^{\frac{1}{p}}.
$$

Using $ |a| - |b| \leq | |a| - |b| | \leq |a-b| $ in \eqref{eq:VortRoughForm} and the bounds \eqref{LBounds} we get 
\begin{align*}
\delta (|\bar{u}|)_{st} & \lesssim    |\bar{u}|_{L^{\infty}([s,t])} (1 + |v|_{L^{\infty}([s,t]; \bH^1)})  \omega_{Z}(s,t)^{\frac{1}{p}} + (1+ |v|_{L^{\infty}([s,t]; \bH^1)})^2  \omega_{Z}(s,t)^{\frac{1}{p}}
\end{align*}

Applying Lemma \ref{lem:RoughGronwall} with $\omega = (1 + |v|_{L^{\infty}([0,T]; \bH^1)})^{p}  \omega_{Z}$ and $\phi(s,t) = (1 + |v|_{L^{\infty}([0,T]; \bH^1)})^2 \omega_{Z}(s,t)^{\frac{1}{p}}$ we get that 
$$
|\bar{u}|_{L^{\infty}([0,T])} \leq 2 \exp \left\{ \frac{\omega(0,T)}{L\alpha} \right\} ( |\bar{u}_0| + K (1 + |v|_{L^{\infty}([0,T]; \bH^1)})^2 \omega_{Z}(0,T)^{\frac{1}{p}} )
$$
and the result follows.
\end{proof}

\section{Enstrophy balance and uniqueness in two spatial dimensions} \label{section:ProofOfTheoremUniqueness}

This section is devoted to the  proof of Theorem \ref{thm:UniqandEnEq2d} and Corollary \ref{cor:stability}, which we split into three parts. First, we establish the enstrophy balance \eqref{eq:EnergyEquality} in Section \ref{sec:Energy}. Second, we prove uniqueness in Section \ref{sec:Uniq}. Thirdly, we show stability in  Section \ref{s:stab}.

Throughout this section, we let $d=2$. In particular, the vorticity $\xi$ is scalar valued and consequently the associated function spaces contain functions that are scalar valued. Since the dimension  will always be clear from the context, we do not alter the notations introduced in Section \ref{ss:notation}.

\subsection{Enstrophy balance} \label{sec:Energy}

In the classical setting, to show \eqref{eq:EnergyEquality}, one would test \eqref{eq:vv} by the solution $\xi$ and use that $u$ and $\sigma_k$ are divergence-free.
Since $\xi^{\natural}$ is not expected to be better behaved than a spatial distribution, one cannot directly test the equation by the solution itself. Instead, we employ a standard trick in PDE theory, namely, the  doubling of variables technique. 
Define the tensor $\xi \otimes \zeta(x,y) := \xi(x) \zeta(y)$, the symmetric tensor $\xi \st \zeta = \frac12 ( \xi \otimes \zeta + \zeta \otimes \xi)$ and the scale of Sobolev spaces $W_{\otimes}^{n,2} := W^{n,2}( \bT^2 \times \bT^2)$.

Variations of the following result have already been proved in \cite{HH17}, \cite{DeGuHoTi16}, \cite{HN} and \cite{HLN}, so we omit the proof.

\begin{proposition}
The mapping $\xi^{\otimes 2} :[0,T] \rightarrow W^{0,2}_{\otimes}$ satisfies the equation
\begin{equation} \label{eq:tensor}
\delta \xi^{\otimes 2}_{st} = \int_s^t 2 \big[  \vartheta \xi_r \st \Delta \xi_r  - \xi_r \st u_r \nabla \xi_r \big] dr + (\Gamma_{st}^1 + \Gamma_{st}^2) \xi_{st}^{\otimes 2} + \xi_{st}^{\otimes 2, \natural} 
\end{equation}
in $W^{-3,2}_{\otimes}$.
Here $(\Gamma^1,\Gamma^2)$ is the unbounded rough driver on $(W^{n,2}_{\otimes})_{n}$ defined by the second quantization 
$$
\Gamma^1:=A^{1}\otimes I + I\otimes A^{1}, \quad \Gamma^2:= A^{2}\otimes I + I \otimes A^{2}+A^{1}\otimes A^{1},
$$
and $\xi^{\otimes 2, \natural} \in C^{ \frac{p}{3}  -\textnormal{var}}_{2, \varpi, L}([0,T]; W^{-3,2}_{\otimes})$, for a control $\varpi$ and $L>0$. 
\end{proposition}

The next step is to test $\xi^{\otimes 2}$ against an approximation of $\delta_{x=y}$ so that we can justify the testing of $\xi$ against itself; that is, to justify the evaluation of $\xi\otimes\xi$ at the diagonal $x=y$. As usual, within our framework, in order to obtain estimates we shall  rewrite the approximation in the standard form \eqref{eq:AbstractURDEquation} and use Theorem~\ref{Thm2.5Abstract}. The two ingredients in Theorem \ref{Thm2.5Abstract}--the scale of spaces and the corresponding family of smoothing operators--will be constructed in this section. The approximation of $\delta_{x=y}$ we shall use will increase the support of the solution, which could pose a problem when working for instance on a bounded domain other than $\bT^{d}$. However, since we are on the torus, the unit ball $B(0,1)$ is a subset of $\bT^d$, so for any integrable and periodic function $f$ we have 
$$
\int_{B(0,1) + \bT^d} |f(x)| dx \leq \int_{2 \bT^d} |f(x)| dx = 2 \int_{\bT^d} |f(x)| dx,
$$
which means that the increase of the support of integration is always a continuous operation on $L^1(\bT^d)$.

Introduce the coordinates $
x_+ := \frac{x + y}{2} $ and $ x_- := \frac{x - y}{2}
$
and denote by $\nabla_{\pm} := \nabla_x \pm \nabla_y$. We consider test functions that are periodic in the $x_{+}$ direction and compactly supported in the $x_{-}$ direction. More precisely, we define the spaces
\begin{equation} \label{def:E-scale}
\mathcal{E}^n_{\nabla} :=  \left\{ \Phi \in W^{n, \infty}( \bR^d \times \bR^d) :  \Phi(x + k 2 \pi, y + k 2\pi) = \Phi(x,y),\; \forall k \in \mathbf{Z}^d, \textrm{ and } |x_-| \geq 1 \Rightarrow \Phi(x,y) = 0 \right\} 
\end{equation}
equipped with the norm 
$$
|\Phi|_{n, \nabla } := \max_{k+l \leq n} \textrm{esssup} \left\{  |\nabla_+^k \nabla_-^l \Phi(x_+ + x_-, x_+ - x_-)| : x_+ \in \bT^d, x_- \in \bR^d \right\}
$$
and the dual pairing between $\mathcal{E}^{-n}_{\nabla}$ and $\mathcal{E}^{n}_{\nabla}$ given by
$$
\langle \Phi, \Psi \rangle_{\nabla} = \int_{\bT^d} \int_{\bR^d} \Phi(x_+ + x_-,  x_+ - x_-) \Psi(x_+ + x_-,  x_+ - x_-) dx_- dx_+ .
$$
Notice that the test functions in \eqref{def:E-scale} are not periodic in the original variables $x,y$ separately, only in $x_{+}$. In addition, due to the compact support of the test functions in the $x_{-}$ variable, the domain of integration in the duality product above can be written in the $(x,y)$-coordinates and is equal to
$$
\Omega = \left\{ (x,y) : x_+ \in \bT^d , x_- \in B(0,1) \right\}.
$$
Define  the blow-up transformation 
$$
T_{\epsilon} \Phi(x,y) = \epsilon^{-d} \Phi \left(x_+ + \frac{x_-}{\epsilon},  x_+ - \frac{x_-}{\epsilon} \right).
$$
and notice that its dual with respect to $\langle \cdot, \cdot \rangle_{\nabla}$ it is given by 
$$
T^*_{\epsilon} \Phi(x,y) = \Phi(x_+ + \epsilon x_-,  x_+ - \epsilon x_-)
$$
and that $T_{\epsilon}^{-1} = \epsilon^d T_{\epsilon}^*$.

We shall need the following uniform estimates.

\begin{lemma} \label{lem:UniformBlowUpEstimates}
For any $f \in \bH^0$ and $g,h \in \bH^1$ we have
\begin{equation} \label{eq:BlowUpL2}
\left| T_{\epsilon}^* (f^{\otimes 2}) \right|_{-0,\nabla}  \lesssim |f|_{L^2(\bT^d)}^2   \quad \textrm{ and } \quad \left| T^*_{\epsilon} (g \otimes \Delta h) \right|_{-1, \nabla}  \lesssim |g|_{1} |h|_{1}  .
\end{equation}
\end{lemma}

\begin{proof}
For $f \in \bH^0$ we have by H\"older's inequality
\begin{align*}
\left| \langle f^{\otimes 2}, T_{\epsilon} \Phi \rangle_{\nabla} \right| & = \left| \int_{\bR^d} \int_{\bT^d} f(x_+ + \epsilon x_-) f(x_+ - \epsilon x_-) \Phi(x_+ + x_-, x_+ - x_-) dx_+ dx_-\right| \\
& \leq \max_{ \tau \in \{-1,1\} } \int_{B(0,1)} \int_{\bT^d} |f(x_+ + \tau \epsilon x_-)|^2 dx_+ \sup_{x_+} |\Phi(x_+ + x_- , x_+ - x_-)| dx_- .
\end{align*}
Introduce the change of variables $z_{\tau} = x_+ + \tau \epsilon x_- \in \bT^d \pm \epsilon B(0,1) \subset \bT^d + B(0,1) $ so that 
$$
\left|\langle f^{\otimes 2}, T_{\epsilon} \Phi \rangle_{\nabla} \right| \leq |B(0,1)| |\Phi|_{0, \nabla} \int_{\bT^d + B(0,1)} |f(z)|^2 dz  \lesssim |\Phi|_{0, \nabla} |f|_{L^2(\bT^d)}^2  ,
$$
which proves the first estimate in \eqref{eq:BlowUpL2}.

For $ g,h \in \bH^1$ we can write
\begin{align*}
g \otimes \Delta h  & = \nabla_y (g \otimes  \nabla h) = (\nabla_+ - \nabla_x) (g \otimes  \nabla h) = \nabla_+ (g \otimes  \nabla h) -  \nabla g \otimes  \nabla h.
\end{align*}
Since $\nabla_+ T_{\epsilon} =  T_{\epsilon} \nabla_+ $ we get 
\begin{align*}
\langle T^*_{\epsilon} g \otimes \Delta h, \Phi \rangle_{\nabla}  
& = - \langle T^*_{\epsilon}  (g \otimes  \nabla h) ,  \nabla_+ \Phi \rangle_{\nabla}  - \langle T_{\epsilon}^*\nabla g \otimes  \nabla h ,  \Phi \rangle_{\nabla} \\
& =  - \int_{\bR^d} \int_{\bT^d} g(x_+ + \epsilon x_-) \nabla h(x_+ - \epsilon x_-)   \nabla_+ \Phi(x_+ + x_-, x_+ - x_-) dx_+ dx_-  \\
& \quad -  \int_{\bR^d} \int_{\bT^d} \nabla g(x_+ + \epsilon x_-) \nabla h(x_+ - \epsilon x_-)   \Phi(x_+ + x_-, x_+ - x_-) dx_+ dx_- .
\end{align*} 
Following a similar derivation of the estimate $\langle f^{\otimes 2}, T_{\epsilon} \Phi \rangle_{\nabla}$, we get
\begin{align*}
|\langle T^*_{\epsilon} g \otimes \Delta h, \Phi \rangle_{\nabla}|  & 
\lesssim |g|_{L^2( \bT^d + B(0,1))} |\nabla h|_{L^2( \bT^d + B(0,1))}  |\Phi|_{1, \nabla} + |\nabla g|_{L^2( \bT^d + B(0,1))} |\nabla h|_{L^2( \bT^d + B(0,1))}  |\Phi|_{0, \nabla} \\
& \lesssim |g|_{1} |h|_{1}  |\Phi|_{1, \nabla} ,
\end{align*}
which proves the second estimate in \eqref{eq:BlowUpL2}.
\end{proof}

The next step in order to be able to apply Theorem \ref{Thm2.5Abstract} is to construct a family of smoothing operators on the scale $(\mathcal{E}_{\nabla}^n)_n$.
Recall that $ \Omega \subset (2 \bT^d) \times (2 \bT^d)$,
so that we may choose a mollifier (in both variables), $\rho_{\eta}$, such that its support is included in $ \Omega$ and we have
$$
J^{\eta} \Phi(x,y) = \int_{\Omega} \Phi(x+u,y+v) \rho_{\eta}(u,v) du dv = \int_{(2 \bT^d) \times (2 \bT^d)} \Phi(x+u,y+v) \rho_{\eta}(u,v) du dv .
$$
It can be checked easily that $J^{\eta}$ acts as a smoothing operator on the scale $W^{n,\infty}( (2 \bT^d) \times (2 \bT^d))$.  We could try to restrict to $\mathcal{E}_{\nabla}^n$, but problem is that our test function space is constructed such that $\Phi(x,y) = 0$ when $|x_-| \geq 1$, and convolution increases this support. However, the increase cannot be too large since
$$
\supp_{x_-}(J^{\eta}\Phi) \subset \supp_{x_-}(\Phi) + \supp(\rho_{\eta})  \subset B(0,1) + B(0, \eta) \subset B(0, 1 + \eta) ,
$$
meaning that our smoothing operator is not well defined as a mapping from $\mathcal{E}_{\nabla}^n$ into itself. We work around this by introducing a function that decreases the support by $\eta$.

\begin{lemma}
There exists a family of smoothing operators $(\bar{J}^{\eta})_{\eta \in [0,1]}$ on the scale $(\mathcal{E}_{\nabla}^n)_{n}$ .
\end{lemma}

\begin{proof}
Let $\theta_{\eta} : \bR \rightarrow [0,1]$ be a smooth function such that 
$$
\theta_{\eta}(\zeta) = \left\{
\begin{array}{ll}
1 & \textrm{ if } |\zeta| \leq 1 - 3 \eta \\
0 & \textrm{ if } |\zeta| \geq 1 - 2 \eta \\
\end{array} \right. ,
$$
and $|\nabla^k \theta_{\eta} |_{\infty} \lesssim \eta^{-k}$ for $k\in \{1,2\}$ and define $\Theta_{\eta}(x,y) = \theta_{\eta}(x_-)$. The following estimates are proved in \cite[Proposition 5.3]{DeGuHoTi16}:
\begin{equation} \label{eq:ThetaBounds}
| (I - \Theta_{\eta}) \Phi |_{0, \nabla} \lesssim \eta^k |  \Phi |_{k, \nabla} , 
\quad
| \Theta_{\eta} \Phi |_{k, \nabla} \lesssim  |  \Phi |_{k, \nabla}
\end{equation}
for $k\in \{0,1,2\}$.
Moreover, for every $\Phi \in \mathcal{E}_{\nabla}^n$ we have 
$$
\supp_{x_-}( \Theta_{\eta} \Phi) \subset B(0,1-\eta),
$$
which yields
$$
\supp_{x_-}(J^{\eta}\Theta_{\eta} \Phi) \subset \supp_{x_-}(\Theta_{\eta} \Phi) + B(0, \eta) \subset B(0, 1 - \eta) + B(0, \eta) \subset  B(0, 1 ) ,
$$
so that $ \bar{J}^{\eta} := J^{\eta}\Theta_{\eta}$ is a well defined operator on $\mathcal{E}_{\nabla}^n$. 

Similarly we have $\supp_{x_-}( (I- J^{\eta})\Phi) \subset \supp_{x_-}(\Phi) + \supp(\rho_{\eta}) $ and $\supp_{x_-}(\Theta_{\eta} \Phi) \subset B(0, 1 - 2 \eta)$ we have that $(I - J^{\eta}) \Theta_{\eta} \Phi \in \mathcal{E}_{\nabla}^n$ for every $\Phi \in \mathcal{E}_{\nabla}^n$.

It remains to show \eqref{smoothingOperator}. The second estimate is obvious. The first follows from the equality
$$
(I - \bar{J}^{\eta})\Phi = (I - J^{\eta})\Phi + J^{\eta}(I - \Theta_{\eta})\Phi .
$$
together with the estimates in \eqref{eq:ThetaBounds}.
\end{proof}

We are now ready to derive the equation for $\xi^2$. To do this, we evaluate \eqref{eq:tensor} in $T_{\epsilon} \Phi$ for any $\Phi \in \mathcal{E}_{\nabla}^3$ to get 
\begin{equation} \label{eq:squareEpsilon}
\delta \xi^{ \epsilon , 2}_{st} = 2 \int_s^t \vartheta T_{\epsilon}^* ( \xi_r \st \Delta \xi_r)  - T_{\epsilon}^*( \xi_r  \st u_r \cdot \nabla \xi_r)  dr + (\Gamma_{st}^{1,\epsilon} + \Gamma_{st}^{2,\epsilon} ) \xi^{ \epsilon , 2}_{s} + \xi^{ \epsilon , 2, \natural}_{st},
\end{equation}
where we have defined 
$$
\xi^{ \epsilon , 2} = T_{\epsilon}^* \xi^{\otimes 2}, \qquad \Gamma_{st}^{i,\epsilon} = T_{\epsilon}^* \Gamma^i T_{\epsilon}^{*,-1}, \qquad \xi^{ \epsilon , 2, \natural} = T_{\epsilon}^* \xi^{\otimes 2,\natural} .
$$

To take the limit as $\epsilon \rightarrow 0$, we apply Theorem \ref{Thm2.5Abstract} to bound the remainder $\xi^{\epsilon,2,\natural}$ in terms of the drift and the unbounded rough driver $(\Gamma^{1, \epsilon},\Gamma^{2, \epsilon})$. Notice that this is possible since the equation is satisfied on the scale $\mathcal{E}_{\nabla}^n$ and we have defined a smoothing operator on this scale. 

The next task is to show that the unbounded rough driver and the drift are uniformly bounded in $\epsilon$. The first part is proven in \cite{HH17}, \cite{DeGuHoTi16} and \cite{BaGu15}, see Proposition \ref{prop:URDRenormalizable} below. The second part will be formulated explicitly in Lemma \ref{lem:DriftRenormalizable} below. 

\begin{proposition} \label{prop:URDRenormalizable}
Assume $\sigma_j \in \W^{3, \infty}$ and  $\mathbf{Z}\in  \mathcal{C}^{p-\textnormal{var}}_g([0,T];\bR^K)$. Then $(\Gamma^{1, \epsilon},\Gamma^{2, \epsilon})_{\epsilon}$ is a bounded family of unbounded rough drivers on $\mathcal{E}_{\nabla}^n$.
Moreover, for $\Phi(x,y) = \psi(x-y) \phi(\frac{x+y}{2})$ where $\psi$ is nonnegative, smooth, has compact support and $\int \psi = 1$ and $\phi \in \W^{3, \infty}$, 
\begin{align*}
\langle \xi^{\epsilon, 2}_s, \Gamma_{st}^{i,\epsilon,*} \Phi \rangle_{\nabla} \rightarrow (\xi^2, A_{st}^{i,*} \phi) 
\end{align*}
\end{proposition}

We now show that the drift is uniformly bounded in $\epsilon$. This allows us to take the limit as $\epsilon \rightarrow 0$ in the approximation of $\delta_{x=y}$.

\begin{lemma} \label{lem:DriftRenormalizable}
There exists a control $\omega$ such that
$$
\left|\int_s^t \vartheta T_{\epsilon}^* ( \xi_r \st \Delta \xi_r)  - T_{\epsilon}^*( \xi_r  \st u_r \cdot \nabla \xi_r)  dr \right|_{-1,\nabla} \leq \omega(s,t) .
$$
Moreover, for $\Phi(x,y) = \psi(x-y)\phi(\frac{x+y}{2})$ where $\psi$ is nonnegative, smooth, has compact support and $\int \psi = 1$ and $\phi \in \W^{3, \infty}$,
\begin{align*}
\int_s^t \langle \vartheta T_{\epsilon}^* &  ( \xi_r \st \Delta \xi_r)  - T_{\epsilon}^*( \xi_r  \st u_r \cdot \nabla \xi_r) , \Phi \rangle_{\nabla} dr  \rightarrow    \int_s^t \vartheta(\xi^2_r,\Delta \phi) -  2 \vartheta (|\nabla \xi_r|^2, \phi)  + (\xi_r^2, \nabla\cdot(u_r \phi))  dr.
\end{align*}
\end{lemma}

\begin{proof}
For the first part, it was already proved in Lemma \ref{lem:UniformBlowUpEstimates} that
$$
\left|\int_s^t   \langle T_{\epsilon}^* (\xi_r \st \Delta \xi_r), \Phi \rangle_{\nabla} dr \right| \lesssim \int_s^t |\nabla \xi_r|_{L^2( \bT^d + B(0,1) )}^2 dr |\Phi|_{1, \nabla}  \lesssim  \int_s^t |\nabla \xi_r|_{0}^2 dr |\Phi|_{1, \nabla} .
$$
Note that the above right-hand-side is a control.  To show that we can also control the non-linear term, consider
\begin{align*}
\big| \langle T_{\epsilon}^*( \xi_r   & \st u_r  \cdot \nabla \xi_r)   , \Phi \rangle_{\nabla}  \big|    
\leq  |\Phi|_{0,\nabla}  \int_{B(0,1)} \int_{\bT^d} |\xi_r(x_+ + \epsilon x_-) u_r(x_+ - \epsilon x_-) \cdot\nabla \xi_r(x_+ - \epsilon x_-)| dx_+ dx_- \\
& \leq |\Phi|_{0,\nabla}  |\xi_r|_{L^4(\bT^d + B(0,1))} |u_r|_{L^4(\bT^d + B(0,1))}|\nabla \xi_r|_{L^2(\bT^d + B(0,1))}   \\
& \lesssim |\Phi|_{0,\nabla}    |\xi_r|_{1}^2 |u_r|_{1} ,
\end{align*}
where we have used the interpolation inequality $|\cdot |_{L^4} \lesssim |\cdot|_1$ and the inequality
$$|\cdot|_{L^p(\bT^d + B(0,1))} \lesssim |\cdot|_{L^p(\bT^d)}, $$ which follows from periodicity,  in the last step. Integrating the above with respect to $r$ over $[s,t]$ gives 
$$
\left| \int_s^t  T_{\epsilon}^* (\xi_r \st u_{r}\cdot\nabla \xi_r)dr \right|_{-0, \nabla} \lesssim \int_s^t  |\xi_r|_{1}^2 |u_r|_{1} dr \leq \sup_r  |u_r|_{1}  \int_s^t |\xi_r|_{1}^2 dr,
$$
which is a control. This shows the first part of the statement.

The second part follows by noticing that for $\Phi(x,y) = \psi(x-y) \phi(\frac{x+y}{2}) $ we have $T_{\epsilon}\Phi(x,y) = \psi_{\epsilon}(2x_-) \phi(x_+)$ where $\psi_{\epsilon}$ converges to a Dirac-delta. In particular, standard arguments show that
$$
\langle \xi_r \otimes \Delta \xi_r, \psi_{\epsilon} \phi \rangle_{\nabla} \rightarrow - (|\nabla \xi_r|_0^2, \phi) - ( (\xi_r \cdot \nabla) \xi_r, \nabla \phi)
$$
and 
$$
\langle \xi_r \otimes  (u_r \cdot\nabla)  \xi_r, \psi_{\epsilon} \phi \rangle_{\nabla} \rightarrow  ( (u_r \cdot\nabla) \xi_r, \xi_r \phi),
$$
for all $r$ such that $\xi_r \in \W^{1,2}$. 
\end{proof}

We are ready to derive the equation for $\xi^2$.

\begin{theorem} \label{theorem:EnergyEquality}
Assume that for each $k\in \{1,\ldots,K\}$, $\sigma_j \in \W^{3, \infty}_{\Div}$ (i.e. divergence-free) and that $\mathbf{Z}\in  \mathcal{C}^{p-\textnormal{var}}_g([0,T];\bR^K)$. Then $\xi^2$ satisfies
\begin{equation} \label{eq:SquaredXi}
\delta \xi_{st}^2(\phi)  =  - \int_s^t   2 \vartheta (|\nabla \xi_r|^2,  \phi) - \vartheta (\xi_r^2, \Delta \phi) - 2 ( \xi_r^2, \nabla\cdot (u_r \phi)) dr + (\xi_s^2, [A_{st}^{1,*} + A_{st}^{2,*}] \phi) + \xi_{st}^{2, \natural}(\phi),
\end{equation}
for every $\phi \in \W^{3, \infty}$. In particular, $\xi \in C_T \bH^0$ and  enstrophy balance \eqref{eq:EnergyEquality} holds.
\end{theorem}

\begin{proof}
From Theorem \ref{Thm2.5Abstract}, we have  $|\xi_{st}^{\epsilon,2,\natural}|_{-3,\nabla} \leq \omega_{2, \natural}(s,t)^{\frac{3}{p}}$ for all $s,t$ such that $\varpi(s,t) \leq L$ for some controls $\omega_{2,\natural}, \varpi$ and $L>0$, which are independent of $\epsilon$. Testing \eqref{eq:squareEpsilon} against $\Phi(x,y) = \psi(2 x_-) \phi(x_+)$ as in Lemma \ref{lem:DriftRenormalizable} and letting $\epsilon \rightarrow 0$ we get \eqref{eq:SquaredXi}. 

Letting $\phi \equiv 1$ in \eqref{eq:SquaredXi} we obtain
$$
\delta (|\xi|_0^2)_{st} + 2 \vartheta \int_s^t |\nabla \xi_r|_0^2  dr  =    (\xi_s, A_{st}^{1,*}1 + A_{st}^{2,*}1) + \xi_{st}^{2, \natural}(1).
$$
Since $\sigma_k$ are divergence-free, we deduce that
$A^{i,*} 1 = 0$, for $i\in \{1,2\},$
which yields
$$
\delta (|\xi|_0^2)_{st} + 2 \vartheta \int_s^t |\nabla \xi_r|_0^2 dr  =  \xi_{st}^{2, \natural}(1) .
$$
Summing the above equality over any partition $\pi = (t_i)_{i=1}^N$ of $[s,t]$ yields
\begin{align*}
\xi_{st}^{2, \natural}(1)& = \delta (|\xi|_0^2)_{st}   + 2 \vartheta \int_s^t |\nabla \xi_r|_0^2 dr = \sum_{ t_i \in \pi}  \left( \delta (|\xi|_0^2)_{t_i t_{i+1} }   + 2 \vartheta \int_{t_i}^{t_{i+1}} |\nabla \xi_r|_0^2 dr \right) \\
& =  \sum_{ t_i \in \pi} \xi_{t_{i-1} t_i}^{2, \natural}(1) \leq \sum_{ t_i \in \pi} \omega_{2, \natural}(t_{i-1},t_i)^{\frac{3}{p}} \leq \omega_{2, \natural}(s,t) \max_{i} \omega_{2, \natural}(t_{i-1},t_i)^{\frac3p - 1} .
\end{align*}
The above right-hand-side converges to 0 as $|\pi| \rightarrow 0$ so that $\xi^{2, \natural} \equiv 0$, proving \eqref{eq:EnergyEquality}. This proves in particular the continuity of $t \mapsto |\xi_t|_0^2$. Combined with the weak continuity $t \mapsto \xi_t$ in $\bH^{-3}$, we find $\xi \in C_T \bH^0$. 
\end{proof}

\begin{remark} \label{Remark:EnergyBoundsForVelocity}
The above shows the enstrophy balance stated in Theorem \ref{thm:UniqandEnEq2d}. The fact that \eqref{VelocityEnergy} is also satisfied can be proved by an application of the Biot-Savart law as well as Lemma~\ref{lem:aPrioriMean}.
\end{remark}

\subsection{Uniqueness} \label{sec:Uniq}
In this section, we prove that in two spatial dimensions, strong solutions of \eqref{eq:RNSDiffForm} are unique. The key idea of the proof is to derive  a formula for the square of the $L^2$-norm of the difference of the vorticity of two arbitrary solutions. Then we show  that the mean of the velocity depends continuously on the mean-free part of the velocity and the initial mean. The formula for the square can be derived in an identical fashion to the enstrophy balance in Section \ref{sec:Energy}.

We start by showing that the mean of the velocity depends continuously on the mean-free part of the velocity and the initial mean. To see this, let $u^{(i)}$, $i\in \{1,2\}$, be two strong solutions starting from the initial conditions $u^{(i)}_{0}\in  \bH^{1}$, respectively. By Remark \ref{Remark:EnergyBoundsForVelocity}, since $u^{(i)}$ are strong solutions, they  satisfy the energy inequality
$$
|u^{(i)}|_{L^{\infty}_T \bH^1}^2 + \int_0^T |\nabla^2 u_r^{(i)}|_0^2 dr \leq F(|u_0^{(i)}|_1) ,\qquad \textrm{ for } \;\; i\in \{1,2\},
$$
for a suitable function $F$ as in \eqref{VelocityEnergy}.

Formally,  $u := u^{(1)} - u^{(2)}$ solves
\begin{equation}\label{eq:VorDiffEq}
\partial_t u + (u^{(1)} \cdot\nabla u^{(1)} - u^{(2)}\cdot \nabla u^{(2)}) = \vartheta \Delta u + [(\sigma_k \cdot\nabla) u  +( \nabla \sigma_k )u]\dot{z}^k  ,
\end{equation}
which is to be understood as
\begin{equation}\label{eq:VorDiffEq1}
\delta u_{st} = \delta \mu^{\Delta}_{st} + \clA_{st}^1 u_s + \clA_{st}^2 u_s + u_{st}^{\natural}
\end{equation}
in the sense of Definition \ref{def:RNSStrongSol}, where $\mu_t^{\Delta}(\phi) = - \int_0^t \vartheta ( \nabla u_r, \nabla \phi) + ( u_r^1 \cdot\nabla u_r^1 -  u_r^2 \cdot\nabla u_r^2, \phi) dr$ for $\phi \in \bH^1$.

Denote by $v^{(i)}$ the mean-free part of $u^{(i)}$,  $v = v^{(1)} - v^{(2)}$, and $\bar{u} =  \bar{u}^1 - \bar{u}^2$. We start by deriving a bound for the mean $\bar u$.

\begin{lemma} \label{lem:MeanContractive}
If $u=u^{(1)}-u^{(2)}$ satisfies \eqref{eq:VorDiffEq1} in the sense of  Definition \ref{def:RNSStrongSol}, then
$$
\sup_{r \leq t} |\bar{u}_{r}|  \lesssim |\bar{u}_0| + \sup_{r \leq t} |v_r|_1 ,
$$
where the proportionality  constant depends on $|u^{(i)}_{0}|_{1}$, $i\in \{1,2\}$.
\end{lemma}
\begin{proof}
Using
\begin{align*}
( u^{(1)} \cdot\nabla u^{(1)} -  u^{(2)} \cdot\nabla u^{(2)}, \phi) & 
= ( u^{(1)} \cdot\nabla u_r, \phi) +  (u_r \cdot\nabla u^{(2)}, \phi) =  ( u^{(1)} \cdot\nabla v_r, \phi) +  (u_r\cdot \nabla u^{(2)}, \phi) ,
\end{align*}
we may estimate the drift as follows:
\begin{align*}
\omega_{\mu^{\Delta}}(s,t) & \lesssim \left( |v|_{L^{\infty}([s,t]; \bH^1)} + |u^{(1)}|_{L^{\infty}([s,t]; \bH^1)} |v|_{L^{\infty}([s,t]; \bH^1)} + |u|_{L^{\infty}([s,t]; \bH^1)} |v^{(2)}|_{L^{\infty}([s,t]; \bH^1)} \right)(t-s) \\
& \lesssim (1 +  |u^{(1)}|_{L^{\infty}([s,t]; \bH^1)}   + |u^{(2)}|_{L^{\infty}([s,t]; \bH^1)} ) ( |v|_{L^{\infty}([s,t]; \bH^1)} + |\bar{u}|_{L^{\infty}([s,t])} )(t-s)  \\
 & \lesssim (1 +  F(|u^{(1)}_0|_1) + F(|u_0^{(2)}|_1) ) ( |v|_{L^{\infty}([s,t]; \bH^1)} +  |\bar{u}|_{L^{\infty}([s,t])} )(t-s)  .
\end{align*}
By the same reasoning as Lemma \ref{lem:aPrioriMean}, we find
\begin{align*}
|\delta \bar{u}_{st}|
& \lesssim  (1 +  F(|u^{(1)}_0|_1) + F(|u_0^2|_1) )  \sup_{s \leq r \leq t} |\bar{u}_s| \omega_{\clA}(s,t)^{\frac{1}{p}} + \phi(s,t),
\end{align*}
where we have defined 
$$
\phi(s,t) =  (1 +  F(|u^{(1)}_0|_1) + F(|u_0^2|_1) )  |v|_{L^{\infty}([s,t]; \bH^1)} \omega_{\clA}(s,t)^{\frac{1}{p}}.
$$
We get the desired bound by applying Lemma \ref{lem:RoughGronwall}.
\end{proof}

Let $\xi^{(i)}=\nabla \times u^{(i)}$, $i\in \{1,2\}$. We now derive an estimate of $|\xi^{(1)} - \xi^{(2)}|_0^2$.

\begin{theorem} \label{thm:contractiveTheorem}
There is a constant $C$ such that for all $t\in [0,T]$,
\begin{equation} \label{eq:Contraction}
\sup_{\theta \leq t} |\xi_{\theta}^{(1)} - \xi_{\theta}^{(2)}|_0^2 +  \int_0^t |\nabla (\xi_{r}^{(1)} - \xi_{r}^{(2)})|_0^2 dr  \leq C \left( |\xi_0^{(1)} - \xi_0^{(2)}|_0^2 + |\bar{u}^{(1)}_0 - \bar{u}^{(2)}_0|^2  \right) \exp\left\{ C |\xi_0^{(2)}|_0^2 \right\} .
\end{equation}
In particular, strong solutions of \eqref{eq:RNSDiffForm} are unique. 
\end{theorem}
\begin{proof}
Formally, $\xi= \xi_{\theta}^{(1)} - \xi_{\theta}^{(2)}$ solves 
\begin{align*}
\partial_t \xi &  + u^{(1)}\cdot \nabla \xi^{(1)} - u^{(2)} \cdot\nabla \xi^{(2)}  = \vartheta \Delta \xi + \sigma_k\cdot \nabla \xi \dot{z}^k.
\end{align*}
In the same way as in Theorem \ref{theorem:EnergyEquality}, we derive 
\begin{equation}
|\xi_t|_0^2 + 2 \vartheta \int_0^t |\nabla \xi_r|_0^2 dr + \int_0^t (u^{(1)}_r \cdot\nabla \xi^{(1)}_r - u^{(2)}_r\cdot \nabla \xi^{(2)}_r, \xi_r) = |\xi_0|_0^2.
\end{equation}
Now, write
\begin{align*}
(u^{(1)}_r \cdot\nabla \xi^{(1)}_r - u^{(2)}_r \cdot\nabla \xi^{(2)}_r, \xi_r) &
= (u^{(1)}_r\cdot \nabla \xi_r  , \xi_r) + (u_r \cdot\nabla \xi^{(2)}_r  , \xi_r)  = - (u_r\cdot \nabla \xi_r , \xi_r^{(2)}),
\end{align*}
which upon applying Young's inequality $ab \leq C_{\epsilon} a^2 + \epsilon b^2$ and ithe nterpolation inequality $|\cdot |_{L^4} \lesssim |\cdot|_1$, yields
\begin{align*}
|\xi_t|_0^2 + &  2 \vartheta \int_0^t |\nabla \xi_r|_0^2 dr = |\xi_0|_0^2 + \int_0^t (u_r\cdot \nabla \xi_r , \xi_r^{(2)}) dr \leq |\xi_0|_0^2 + \int_0^t |u_r|_{L^4} |\xi^{(2)}_r|_{L^4} |\nabla \xi_r|_0  dr \\
& \lesssim |\xi_0|_0^2 + \int_0^t |u_r|_{1} |\xi^{(2)}_r|_{1} |\nabla \xi_r|_0  dr \leq |\xi_0|_0^2 + C_{\epsilon} \int_0^t |u_r|_{1}^2 |\xi^{(2)}_r|_{1}^2 dr +  \epsilon \int_0^t |\nabla \xi_r|_0^2  dr.
\end{align*}
For $\epsilon$ small enough depending only on $\vartheta$, we get
\begin{align*}
|\xi_t|_0^2 + &  \vartheta \int_0^t |\nabla \xi_r|_0^2 dr \leq |\xi_0|_0^2 + C_{\epsilon} \int_0^t |u_r|_{1}^2 |\xi^{(2)}_r|_{1}^2 dr.
\end{align*}
Using the Biot-Savart law and Lemma \ref{lem:MeanContractive}, we find
\begin{equation} \label{eq:MeanDifference}
|u_r|_1^2  = |\bar{u}_r + v_r |_0^2 + |\nabla v_r|^2_0 \lesssim |\bar{u}_0|^2 + \sup_{\theta \leq r} |v_{\theta }|_{1}^2  \lesssim  |\bar{u}_0|^2 +  \sup_{\theta \leq r} | \xi_{\theta }|_{0}^2,
\end{equation}
which gives
\begin{align*}
\sup_{\theta \leq t} |\xi_{\theta}|_0^2 + &  \vartheta \int_0^t |\nabla \xi_r|_0^2 dr \leq C \left( |\xi_0|_0^2 + |\bar{u}_0|^2 \int_0^t  |\xi^{(2)}_r|_{1}^2 dr + \int_0^t  |\xi^{(2)}_r|_{1}^2 \sup_{\theta \leq r} |\xi_{\theta}|_0^2 dr \right). 
\end{align*}
Gronwall's inequality then implies
\begin{align*}
\sup_{\theta \leq t} |\xi_{\theta}|_0^2 + 2 \vartheta \int_0^t |\nabla \xi_r|_0^2 dr  & \leq C \left( |\xi_0|_0^2 + |\bar{u}_0|^2 \int_0^t  |\xi^{(2)}_r|_{1}^2 dr \right) \exp\left\{  C \int_0^t  |\xi^{(2)}_r|_{1}^2  dr \right\} \\
& \leq C \left( |\xi_0|_0^2 + |\bar{u}_0|^2  \right) \exp\left\{  C |\xi_0^{(2)}|_0^2 \right\},
\end{align*}
where we have used Theorem \ref{theorem:EnergyEquality} for $\xi^{(2)}$ in the last inequality, and the constant $C$ may vary from line to line. This shows \eqref{eq:Contraction}.

To see that this implies uniqueness of \eqref{eq:RNSDiffFormSys}, assume $u^{(1)}_0 = u^{(2)}_0$. From \eqref{eq:Contraction} we get that $\xi^{(1)} = \xi^{(2)}$. From \eqref{eq:MeanDifference} we find that $u^{(1)} = u^{(2)}$. 
\end{proof}

\subsection{Stability}
\label{s:stab}

In this section, we prove Corollary \ref{cor:stability}. Since it is similar to the proof of (forthcoming) Theorem \ref{existenceThm}, we only sketch the main steps here.

\begin{proof}[Proof of Corollary \ref{cor:stability}]
Consider a sequence $(u^n_0, \sigma^n, \bZ^n)_{n \geq 1} \in \bH^1\times (\W^{3, \infty}_{\Div}\big)^K \times \mathcal{C}^{p-\textnormal{var}}_g([0,T];\bR^K)$ converging to some element $(u_0, \sigma, \bZ)$ in this space. From Theorem \ref{thm:UniqandEnEq2d} we have
$$
\sup_{ t \in [0,T]} |\xi_t^n|^2_0 +2\vartheta \int_0^T |\nabla \xi^n_r|_0^2\,dr= |\xi_0^n|_0^2 . 
$$
As in the proof of Theorem \ref{existenceThm}, we can deduce that $\{u^n\}$ remains in a bounded set of $L_T^2 \bH^2 \cap L_T^{\infty}\bH^1 \cap C^{p-\textnormal{var}}([0,T]; \bH^{0})$, and thus there exists a subsequence, $\{u^{n_k}\} $ converging to some $u$ in $C_T \bH^0 \cap L_T^2 \bH^1$. Moreover, by the assumptions on $(\sigma^n, \bZ^n)_{n \geq 1}$,  the corresponding unbounded rough drivers, denoted by $(\clA^{n,1}, \clA^{n,2})$, converge to $(\clA^{1}, \clA^{2})$ in the strong topology; that is, $\clA^{n,i}$ converges to $\clA^i$ in the strong topology of $\mathcal{L}(\bH^{k}, \bH^{k-i})$ for $i\in \{1,2\}$. Taking the limit gives that $u$ satisfies \eqref{SystemSolutionU}. By uniqueness of solutions in dimension two, Theorem~\ref{thm:UniqandEnEq2d}, we get that the full sequence $\{u^n\}$ must converge, thus showing continuity of the solution map.

Suppose now that $B$ is a Brownian motion and let $B^n$ denote a piecewise linear approximation of $B$. It is well known that $(B^n, \mathbb{B}^n)$ converges $\mathbb P$-a.s. in the rough path topology to $(B, \mathbb{B})$ where $\mathbb{B}_{st}^{i,j} := \int_s^t B_{sr}^i \circ dB^j_r$ is the Stratonovich integral. For a fixed $\phi \in \bH^2$, we have as in \cite[Corollary 5.2]{FrHa14} that the rough path integral $\int_0^{\cdot} (u_r, (\nabla \sigma_k)\phi -\Div(\sigma_k \phi)) d \mathbf{B}_r$ and the Stratonovich integral $\int_0^{\cdot} (u_r, (\nabla \sigma_k)\phi - \Div(\sigma_k \phi)) \circ d B_r$ coincide on a set, $\Omega_{\phi}$, of full measure. Choosing a dense subset $\{\phi_l\}_{l \in \mathbb{N}}$ of $\bH^2$ and letting $\Omega_0 := \cap_{l \in \mathbb{N}} \Omega_{\phi_l}$ we see that the solutions must agree on $\Omega_0$.
From the above continuity, we obtain the claimed Wong-Zakai result. 
\end{proof}

\section{Existence}
\label{Section:Galerkin}

In this section, we establish existence of a strong solution as formulated in Theorem \ref{existenceThmNoGalerkin} based on a  Galerkin approximation.
For $d\in \{2,3\}$, let $\{h_n\}_{n=0}^{\infty}$ be the smooth eigenfunctions of the Stokes operator $
-P\Delta$ on $\bT^d$ with corresponding eigenvalues $\{\lambda_n\}_{n=0}^{\infty}$ where $\lambda_{0}=0$ (corresponding to $h_{0}\equiv \rm{const}$) and $\lambda_{n}>0$ for $n\in\bN$.
We choose the eigenfunctions $\{h_n\}_{n=0}^{\infty}$ such that they form an orthonormal basis of $\bH^0$ and an orthogonal basis of $\bH^1$. For a given $n\in \bN$, define $l_n=\lambda_n^{-\frac{1}{2}}\nabla \times h_n$. It can easily be verified that $(\nabla^{\perp}f,g) = (f,\nabla \times g)$ in $d=2$ and $(\nabla \times f, g)= (f,\nabla \times g)$ in $d=3$. Thus $\{l_n\}_{n=1}^{\infty}$ forms an orthonormal basis of $\dot{\bH}^0$ and we have
\begin{equation}\label{eq:projectionidentity}
(f,h_n)\nabla \times h_n=(\nabla \times f  ,l_n)l_n.
\end{equation}

For a given $N\in \mathbf{N}$, let
$\bH_N=\operatorname{span}(\{h_{0},h_1,\ldots,h_N\})$ and $\bL_N=\operatorname{span}(\{l_1,\ldots,l_N\})$, and define   $P_N: \bH^{-1}\rightarrow \bH_N$ and $L_N: \dot{\bH}^{-1}\rightarrow \dot{\bL}_N$  by 
$$P_N\varv:=\sum_{n=0}^N(\varv,h_n) h_n, \quad L_N\varv:=\sum_{n=1}^N(\varv,l_n) l_n, \quad  \varv\in \bH^{-1}.$$
It follows from \eqref{eq:projectionidentity} that
\begin{equation}\label{eq:mainprojectionidentity}
\nabla \times P_N\varv =L_N (\nabla \times \varv), \quad \forall \varv\in \bH^{-1}.
\end{equation}

Since $\mathbf{Z}\in C_g^{p-\textnormal{var}}([0,T];\bR^K)$ is a geometric rough path, there is a sequence  of $\bR^K$-valued smooth paths  $\{z^{N}\}_{N=1}^{\infty}$ such that their canonical lifts $\bZ^N=(Z^N,\mathbb{Z}^N)$ converge to $\bZ$ in the rough path topology.  We assume that
\begin{equation} \label{ineq:RPBounds}
|Z_{st}^{N} | \lesssim \omega_Z(s,t)^{\frac{1}{p}},\quad |\mathbb{Z}^{N}_{st}| \lesssim \omega_Z(s,t)^{\frac{2}{p}}, \quad \forall (s,t)\in \Delta_T.
\end{equation}
For convenience, we denote by $N_0$ a constant that bounds  $\sigma=(\sigma_1,\ldots,\sigma_K)$ and its derivatives up to order two. 

The following $N$-th order Galerkin approximation of \eqref{eq:RNSDiffForm}
\begin{equation}\label{eq:SmoothNSDiffForm}
\partial_t u^N + P_NB_P(u^N) =  \vartheta P_N \Delta u^N+  \sum_{k=1}^{K} P_N P[(\sigma_k\cdot \nabla) u^N+(\nabla\sigma_k)u^N]  \dot{z}^{N,k}_t,
\end{equation}
with initial condition $u^N(0)=P_N u_0$
gives a system of ODEs with locally Lipschitz coefficients. Consequently there exists a time interval $[0,T_N)$, for some $T_N>0$ and a unique solution 
$u^N$
of \eqref{eq:SmoothNSDiffForm} on the time  interval $[0,T_N)$. 

Integrating \eqref{eq:SmoothNSDiffForm} over the interval $[s,t]$ we find
\begin{equation} \label{eq:URDGalerkinApprox}
\delta u^N_{st} =  \int_s^t \left( \vartheta P_N \Delta  u^N_r - P_N B_P(u_r^N) \right) \, dr +  \clA_{st}^{N,1} u_s^{N} +  \clA_{st}^{N,2} u_s^{N} + u_{st}^{N, \natural} ,
\end{equation}
where    $\tilde{P}_N := P_N P$, and $\clA^{N,i}_{st}$ and $u_{st}^{N, \natural}$ are defined as in \eqref{eq:URDdef} and \eqref{eq:ExplicitRemainder} respectively with $P$ replaced by $\tilde{P}_N$ and $\bZ$ replaced by $\bZ^N$.

Owing to \eqref{A1bound} and \eqref{ineq:RPBounds}, we have that $(\clA^{N,1}, \clA^{N,2} )$ is uniformly bounded in $N$ as a family of unbounded rough drivers on the scale $( \bH^{m} )_{m }$. That is, there exists a control $\omega_{\clA^N}$ such that \eqref{ineq:UBRcontrolestimates} holds and  for all $(s,t)\in \Delta_T$,
\begin{equation} \label{UniformGalerkinURD}
\omega_{\clA^N}(s,t)\lesssim_{N_0}  \omega_{Z}(s,t).
\end{equation}

It is straightforward to check that    $u^{N, \natural} \in C_2^{\frac{p}{3} -\textnormal{var}}([0,T_N); \bH_N)$ by estimating term-by-term; one makes use of  \eqref{A1bound},  \eqref{trilinear form estimate}, and that $u^N$ is smooth in space and $z^N$ is smooth in time. For given $(s,t)\in \Delta_{T_N}$, let $\omega_{N,\natural}(s,t) := |u^{N,\natural}|^{\frac{p}{3}}_{\frac{p}{3} -\textnormal{var};[s,t];\bH^{-2}}$.
Arguing as in Lemma \ref{Thm2.5}, we find that there is an $L>0$ such that for all $(s,t)\in \Delta_{T_N}$ with  $\omega_{Z}(s,t)\le L$, 
\begin{equation} \label{eq:GalerkinRemainder}
\omega_{N, \natural}(s,t)  \lesssim_p|u^N|^{\frac{p}{3}}_{L_T^{\infty} \bH^1} \omega_{\clA^N}(s,t)  + ( 1 +|u^N|_{L_T^{\infty} \bH^1})^{\frac{2p}{3}}  (t-s)^{\frac{p}{3}} \omega_{\clA^N}(s,t)^{\frac{1}{3}} . 
\end{equation}

In order to obtain uniform energy bounds on $u^N$, we first derive the equation for the vorticity $\xi^N:=\nabla \times u^N$.  Let
$$
\clL_{v}\phi = (v\cdot \nabla) \phi -\mathbf{1}_{d=3} ( \nabla v )\phi, \quad \clL^N_{v}=L_N\clL_{v}.
$$
Using properties of the curl operator in Section \ref{ss:notation} and
\eqref{eq:mainprojectionidentity}, we find that $\xi^N$ satisfies
\begin{equation}\label{eq:SmoothVor}
\partial_t \xi^N +\clL_{u^N}^N\xi^N=  \vartheta L_N \Delta \xi^N+ \clL_{\sigma_k}^N\xi^N \dot{z}^{N,k}_t.
\end{equation}

Obtaining uniform bounds in dimension two is the simplest, due to the conservative nature of  the equation.   However, this is no longer possible in dimension three. Indeed, there is an additional stretching term in the drift and a lower order term in the noise which forces us to use a non-linear version of the rough Gronwall's inequality, Lemma \ref{lem:RoughBihari}.

Let us  begin with the case $d=2$. Testing \eqref{eq:SmoothVor} by $\xi^N$  and using \eqref{eq:B prop}, integration by parts and the fact that  $\Div \sigma_k=0$, for all $k\in \{1,\ldots,K\}$, we obtain 
\begin{align*}
|\xi^N_t|_0^2 +  2 \vartheta\int_0^t |\nabla \xi^N_s|_0^2\, ds& =  |L_N \xi_0|_0^2 - 2 \int_0^t ((u^N_s\cdot \nabla)\xi^N_s,\xi^N_s)\,ds  +  2\int_0^t ((\sigma_k \cdot\nabla) \xi^N_s, \xi^N_s)dz_s^{N,k} \\
&= |L_N\xi_0|_0^2 \leq |\xi_0|_0^2, \quad \forall t\in [0,T_N).
\end{align*}
Let $$v^N=u^N-\bar{u}^N=u^N-\int_{\bT^d} u^Ndx.$$
Owing to the Poincar\'e inequality and \eqref{ineq:BSGrad} we have $
|v^N|_0\lesssim |\nabla v^N|_0=|\xi^N|_0,$ and $
|\nabla^2v^N|_0 = |\nabla \xi^N|_0,
$
hence
\begin{align*}
| v^N_t|_1^2 +  2 \vartheta\int_0^t |\nabla^2 u^N_s|_0^2\, ds \lesssim |\nabla u_0|_0^2, \quad \forall t\in [0,T_N).
\end{align*}
We obtain bounds on $\sup_{t\in [0,T_N)}|\bar{u}^N|$  as in Lemma \ref{lem:aPrioriMean}, giving \eqref{VelocityEnergy} with $u$ replaced by $u^N$ and $T^*$ replaced by $T_N$. For a general time $T$ it is standard to extend the solution to $[0,T]$.

Let us  turn our attention to dimension three. It is not possible to obtain an enstrophy bound independent of the noise approximation like we did in two-dimensions since the noise is not enstrophy conservative due to the presence of the term $(\nabla\sigma_k )\xi^N$. 

Integrating  \eqref{eq:SmoothVor} over the interval $[s,t]$ 
we find 
\begin{equation} \label{eq:RoughSmoothVor}
\delta \xi^N_{st} = \delta \gamma_{st}^N +  [A_{st}^{N,1} +A_{st}^{N,2}]\xi_s^{N}   + \xi_{st}^{N, \natural} ,
\end{equation}
where  
$$
\gamma_t^N : =  \int_0^t \left( \vartheta L_N\Delta \xi_r^N  - \clL_{u^N_r}\xi^N_r \right) \, dr,\quad A^{N,1}_{st} \phi :=   \clL^N_{\sigma_k}\phi  Z_{st}^{N,k},\quad A_{st}^{N,2} \phi  :=  \clL^N_{\sigma_k}  \clL^N_{\sigma_l} \phi\mathbb{Z}^{l,k}_{st},
$$
and
\begin{align*}
\xi_{st}^{N, \natural} & :=   \int_s^t\clL^N_{\sigma_k}\delta \gamma_{sr}^N\, dz_r^{N,k}  +  \int_s^t \int_s^r   \clL^N_{\sigma_k} \clL^N_{\sigma_l} \left[\delta \mu^N_{s r_{1}}+ \int_{s}^{r}\clL^N_{\sigma_m}\xi_{r_2}^N\,dz_{r_1}^{N,m}\right] \,dz_{r_1}^{N,l}dz_r^{N,k}.
\end{align*}

We proceed by deriving the equation for 
$$
\Xi^N=\xi^N\otimes \xi^N=[\xi^{N,i}\xi^{N,j}]_{1\le i,j\le d}.
$$
Defining the symmetric tensor $a\st b := \frac12 (a\otimes b+b\otimes a)$, we have
$
\delta \Xi^N_{st} =2\xi^N_s \st \delta \xi^N_{st}+ (\delta\xi^N_{st})^{\otimes2},
$
and hence
\begin{align} \label{eq:URDSqGalerkinApproxdiff}
\delta\Xi^N_{st}&= \delta \Pi^N_{st}+ [\Gamma^{N,1}_{st}+\Gamma^{N,2}_{st}]\Xi^{N}_s+ \Xi^{N,\natural}_{st},
\end{align}
where
\begin{gather*}
\Pi_{t}^N:=2\int_{0}^t \xi^N_r\st  \left( \vartheta L_N\Delta \xi_r^N  - \clL_{u^N_r}\xi^N_r \right) dr, \quad 
\Gamma_{st}^1:=2A^{N,1}_{st}\hat{\otimes} I, \quad \Gamma_{st}^2:=2A_{st}^{N,2}\hat{\otimes} I+A_{st}^{N,1}\hat{\otimes} A_{st}^{N,1},\\
\Xi^{N,\natural}_{st}=2\xi^N_s\st \xi_{st}^{N, \natural}-2\int_{s}^t\delta \xi^N_{sr}\st \left( \vartheta L_N\Delta \xi_r^N  - \clL_{u^N_r}\xi^N_r \right) dr
+(\delta\xi^N_{st})^{\otimes2}-A^{N,1}_{st}\hat{\otimes} A^{N,1}_{st} \Xi^N_s.
\end{gather*}
By virtue of \eqref{A1bound} and \eqref{ineq:RPBounds}, we have that $(\Gamma^{N,1}, \Gamma^{N,2} )$ is uniformly bounded in $N$ as a family of unbounded rough drivers on the scale $( \dot{\bW}^{m,\infty}(\bT^3;\bR^{3\times 3}) )_{m }$. That is, there exists a control $\omega_{\Gamma^N}$ such that \eqref{ineq:UBRcontrolestimates} holds and  for all $(s,t)\in \Delta_T$,
$$
\omega_{\Gamma^N}(s,t)\lesssim_{N_0}  \omega_{Z}(s,t).
$$
Let us denote by $|\cdot|_{m,\infty}$ the norm on ${\bW}^{m,\infty}(\bT^3;\bR^{3\times 3})$ and for notational simplicity $|\cdot|_{\infty}=|\cdot|_{0,\infty}$.
To find a control for $\Pi^N$, we need to estimate
$$
\Pi^N_{st}(\Phi)=2\int_s^t \xi^N_r\st  \vartheta L_N\Delta \xi_r^N  (\Phi)dr -2\int_s^t \xi^N_r\st  \clL_{u^N_r}\xi^N_r  (\Phi)dr =: I + II.
$$
Applying Young's inequality, we find
\begin{align*}
I&=-2\vartheta \int_s^t \int_{\bT^3}\partial_l\xi^{N}_r \st  \partial_l \xi_r^{N} (\Phi)dx dr-2\vartheta \int_s^t \int_{\bT^3}\xi^{N}_r  \st \partial_l \xi_r^{N}( \partial_l \Phi)dx dr\\
&\le   2\vartheta|\nabla \Phi|_{\infty} \left( \int_s^t|\nabla \xi_r^N|^2_0dr+\int_s^t\sup_{s\le r'\le r}| \xi_{r'}^N|^2_0dr\right).
\end{align*}
We split $II$ into two quantities $III$ and $IV$ and then estimate them separately:
$$
- II=2\int_s^t \int_{\bT^3}\xi^{N}_r\st [(u^{N}_r\cdot \nabla )\xi^{N}_r] (\Phi) dxdr+\int_s^t \int_{\bT^3}\xi^{N}_r\st [(\nabla u^{N}_r)\xi^{N}_r] (\Phi)dxdr =: III + IV.
$$
Using the interpolation inequality $|f|_{L^4} \leq C_3 |f|_0^{\frac{1}{4}} |\nabla f|_0^{\frac{3}{4}}$ for $d=3$, H\"older's and Young's inequality, Lemma \ref{lem:aPrioriMean}, and \eqref{ineq:BSGrad} we get
\begin{align*}
III&\lesssim |\Phi|_{\infty}\int_s^t|\xi^{N}_r|_{0}^{\frac{1}{4}}|\nabla \xi^{N}_r|_{0}^{\frac{3}{4}}|u^{N}_r|_{0}^{\frac{1}{4}}|\nabla u^{N}_r|_{0}^{\frac{3}{4}}|\nabla \xi^{N}_r|_{0} dr \lesssim |\Phi|_{\infty}\int_s^t|\xi^{N}_r|_{0}|u^{N}_r|_{0}^{\frac{1}{4}}|\nabla \xi^{N}_r|_{0}^{7/4} dr\\
&\lesssim  |\Phi|_{\infty}\left(\int_s^t|u^{N}_r|_{0}^{2}|\xi^{N}_r|_{0}^8dr + \int_{s}^t|\nabla \xi^{N}_r|_{0}^{2} dr\right)\lesssim  |\Phi|_{\infty}\left(\int_s^t|\bar{u}^{N}_r|_{0}^{2}|\xi^{N}_r|_{0}^8dr +\int_s^t|\xi^{N}_r|_{0}^{10}dr+ \int_{s}^t|\nabla \xi^{N}_r|_{0}^{2} dr\right)\\
&\lesssim  |\Phi|_{\infty}\left(\int_s^tw_1\left(\sup_{s\le r'\le r}|\xi^N_{r'}|_0\right)dr+ \int_{s}^t|\nabla \xi^{N}_r|_{0}^{2} dr\right),
\end{align*}
where
$$
w_1(y):=(1 +   |\bar{u}_0|) \exp \left\{ C(1 + y)^p \right\} y^{8}+y^{10} .
$$
Similarly,
\begin{align*}
IV&\lesssim |\Phi|_{\infty}\int_s^t \int_{\bT^3}|\xi^{N}_r|_0^{\frac{1}{2}}|\nabla \xi^{N}_r|_0^{\frac{6}{4}}|\nabla u^{N}_r|_0dr dx\lesssim |\Phi|_{\infty}\int_s^t |\xi^{N}_r|_0^{\frac{3}{2}}|\nabla \xi^{N}_r|_0^{\frac{6}{4}}dr \\
&\lesssim   |\Phi|_{\infty}\left(\int_s^t \sup_{s\le r'\le r}|\xi^{N}_{r'}|_0^{6}dr +\int_s^t |\nabla \xi^{N}_r|^{2}_0dr\right).
\end{align*}
Therefore, 
$$
|\Pi_{st}^N|_{-1,\infty}\lesssim \omega_{\Pi^N}(s,t):=\int_s^t|\nabla \xi_r^N|^2_0dr+\int_s^t w_2\left(\sup_{s\le r'\le r}|\xi^{N}_{r'}|_0\right)dr,
$$
where
$$
w_2(y)=(1 +|\bar{u}_0|) \exp \left\{ C(1 + y)^p \right\} y^{8}+y^{10} +y^6+y^2.
$$

Using Theorem \ref{Thm2.5Abstract}, we obtain
\begin{align*}
|\Xi^{N,\natural}_{st}|_{-3,\infty}&\lesssim \sup_{s\le r\le t}|\Xi_r^N|_{-0,\infty} \omega_{Z}(s,t)^{\frac{3}{p}}+\omega_{\Pi^N}(s,t)\omega_{Z}(s,t)^{\frac{1}{p}}\\
&\lesssim \sup_{s\le r\le t}|\xi^N_r|_{0}^2 \omega_{Z}(s,t)^{\frac{3}{p}}+\omega_{Z}(s,t)^{\frac{1}{p}}\int_s^t|\nabla \xi_r^N|^2_0dr+\omega_{Z}(s,t)^{\frac{1}{p}}\int_s^t {w_{2}}\left(\sup_{s\le r'\le r}|\xi^{N}_{r'}|_0\right)dr.
\end{align*}

Testing \eqref{eq:URDSqGalerkinApproxdiff} against the $3 \times 3$ identity matrix, $I_3$, we find
\begin{align*}
\delta (|\xi^N|_{0}^2)_{st}&= \delta \Pi^N_{st}(I_3)+\Xi^N_s( [\Gamma^{N,1,*}_{st}+\Gamma^{N,2,*}_{st}](I_3))+ \Xi^{N,\natural}_{st}(I_3).
\end{align*}
Clearly,
$$
|\Xi^N_s( [\Gamma^{N,1,*}_{st}+\Gamma^{N,2,*}_{st}](I_3))|\le \sup_{s\le r\le t}|\xi^{N}_r|_0^2\omega_{Z}(s,t)^{\frac{1}{p}}.
$$
Moreover, using H\"{o}lder's inequality, the interpolation inequality $|f|_{L^4} \leq C_3 |f|_0^{\frac{1}{4}} |\nabla f|_0^{\frac{3}{4}}$ for $d=3$,  and Young's inequality $ab \leq C_{\epsilon} a^4 + \epsilon b^{4/3}$ for $\epsilon\in (0,1)$ to be determined later, we get 
\begin{align*}
\delta \Pi^N_{st}(I_3)&=-2\vartheta\int_s^t|\nabla \xi^N_r|_0^2dr+\int_s^t \int_{\bT^3}\xi^{N,i}_r\partial_iu^{N,j}_r \xi^{N,j}_rdxdr    \leq -2\vartheta\int_s^t|\nabla \xi^N_r|_0^2dr+\int_s^t |\xi^{N}_r|_{L^4}^2 |\nabla u^{N}_r|_0 dr \\
& \leq -2\vartheta\int_s^t|\nabla \xi^N_r|_0^2dr+ C_3 \int_s^t |\xi^{N}_r|_{0}^{\frac{3}{2}} |\nabla \xi^{N}_r|_0^{\frac{3}{2}} dr  \\
&  \leq -(2\vartheta- C_3 \epsilon)\int_s^t|\nabla \xi^N_r|_0^2dr+C_{\epsilon}C_3\int_s^t \sup_{s\le r'\le r}|\xi^{N}_{r'}|_0^{6}dr.
\end{align*}
Putting it all together we arrive at 
\begin{align*}
\delta (|\xi^N|_{0}^2)_{st} & \leq  \left[ C \omega_Z(s,t)^{\frac{1}{p}}-(2\vartheta-C_3 \epsilon)\right]\int_s^t|\nabla \xi^N_r|_0^2dr+ C \int_s^tw\left(\sup_{s\le r'\le r}|\xi^{N}_{r'}|_0\right)dr+ C \sup_{s\le r\le t}|\xi_r|_0^2\omega_{Z}(s,t)^{\frac{1}{p}},
\end{align*}
where 
$$
w(y)=(1 +|\bar{u}_0|) \exp \left\{ C(1 + y)^p \right\} y^{8}+y^{10} +(1+C_3C_{\epsilon})y^6+y^2.
$$
For $(s,t)\in \Delta_{T_N}$ and $\epsilon>0$ such  $C \omega_Z(s,t)^{\frac{1}{p}}-(2\vartheta- C_3\epsilon)\le - \vartheta$, we have 
\begin{gather*}
\delta (|\xi^N|_{0}^2)_{st}+\vartheta \int_s^t|\nabla \xi^N_r|_0^2dr \lesssim \int_s^tw\left(\sup_{s\le r'\le r}|\xi^{N}_{r'}|_0\right)dr+\sup_{s\le r\le t}|\xi_r|_0^2\omega_{Z}(s,t)^{\frac{1}{p}}.
\end{gather*}
Applying Lemma \ref{lem:RoughBihari} on $[0,T_N)$, we get 
$$
\sup_{0\le r \le T^* \wedge T_N}|\xi^N_r|_{0}^2 \leq W^{-1} \left( W( q|\xi^N_0|_{0}^2 ) + T^* q \right)
$$
is such that 
$$
W(q |\xi^N_0|_{0}^2) + T^*q \in \operatorname{Dom}( W^{-1})
$$
and $W,q,T^{*}$ is given as in Lemma  \ref{lem:RoughBihari}.
Note that $T^{*}$ can be chosen independently of $N$.
It follows that we may extend $\xi^N$ up to time $T^*$ and the $\bH^0$-norm of $\xi^N$ is bounded up to time $T^*$. Hence $u^N$ does not blow up in the time interval $[0,T^*]$ and is a solution of \eqref{eq:SmoothNSDiffForm} on the full time interval. Moreover $\xi^N\in C_{T^*}\dot\bH^0\cap L^2_{T^*}\dot\bH^1$ solves \eqref{eq:SmoothVor}.

We are now ready to send $N$ to $\infty$ in \eqref{eq:URDGalerkinApprox}, which in particular proves Theorem \ref{existenceThmNoGalerkin}. The details are given in the following result. Throughout the rest of this section, let $T^{*}=T$ if $d=2$ whereas for $d=3$ let $T^{*}>0$ be the final time obtained by means of Lemma \ref{lem:RoughBihari} above. 

\begin{theorem} \label{existenceThm}
There exists a subsequence of $\{ u^N \}_{N=1}^{\infty}$ that converges  weakly in $L^2_{T^{*}}\bH^2$, weak-* in $L^{\infty}_{T^{*}}\bH^1$, and strongly in $L_{T^{*}}^{2}\bH^1 \cap C_{T^{*}}\bH^{0}$  to a solution of \eqref{SystemSolutionU} that is weakly continuous in $\bH^1$.
\end{theorem}
\begin{proof}

From the Biot-Savart law, we see that $v^N = \BS \xi^N$ remains in a bounded set in $L^2_{T^{*}}\bH^2\cap L^{\infty}_{T^{*}}\bH^1$. Moreover, from Lemma \ref{lem:aPrioriMean} we get that 
\begin{equation} \label{UniformGalerkinEnergy}
|\bar{u}^N|_{L^{\infty}_{T^{*}} \bR^d} \leq C \exp \{ C(1 + |\xi_0|_0)^p\} (1 + |\bar{u}^N_0|) \leq C \exp \{ C(1 + |\xi_0|_0)^p\} (1 + |u_0|_0),
\end{equation}
where the last inequality comes from $|\bar{u}^N_0| \leq |u_0|_0 $. This gives that $u^N = v^N + \bar{u}^N$ also remains in a bounded set in $L^2_{T^{*}}\bH^2\cap L^{\infty}_{T^{*}}\bH^1$, and we have
$$
|u^N|_{L^{\infty}_{T^{*}} \bH^1} + |u^N|_{L^{2}_{T^{*}} \bH^2} \leq f(|u_0|_1)
$$
for some continuous function $f$.
An application of Banach-Alaoglu yields a subsequence, which we will relabel as $\{u^N\}_{n=1}^{\infty}$, that converges   weakly in $L^2_{T^{*}}\bH^2$ and weak-* in $L^{\infty}_{T^{*}}\bH^1$. To obtain a further subsequence that converges strongly in $L_{T^{*}}^{2}\bH^1 \cap C_{T^{*}}\bH^{0}$, we shall  apply  Lemma A.2 \cite{HLN}; that is, we shall show there exist controls $\omega$ and $\bar{\omega}$ and $L,\kappa>0$  independent of $N$ such that  $|\delta u_{st}^N|_{0}\le \omega(s,t)^{\kappa}$ for all $(s,t)\in \Delta_{T^{*}}$ with $\bar{\omega}(s,t)\le L$. 
From Lemma \ref{AprioriVariation} we obtain using \eqref{eq:GalerkinRemainder}
\begin{align}
|\delta u^N_{st}|_{0} &  \lesssim F(|u_0|_1) ( \omega_{Z}(s,t) + (t-s) )^{ \frac{1}{p} }  \label{uniformSolution}  .
\end{align}
where we have used \eqref{UniformGalerkinURD} and $F$ is some continuous function coming from \eqref{UniformGalerkinEnergy} and \eqref{eq:GalerkinRemainder} combined with \eqref{UVariationWithoutMu} applied to $u^{N}$.

By the compact embedding from Lemma A.2 in \cite{HLN}, there is a subsequence of  $\{u^{N}\}_{N=1}^{\infty}$, which we keep denoting by $\{u^{N}\}_{N=1}^{\infty}$, converging strongly to an element $u $ in $C_{T^{*}}\bH^{0}\cap L^2_{T^{*}}\bH^1$. Furthermore, owing to Lemma A.3 in \cite{HLN}, we know that $u$ is continuous with values in $\bH^0_w$ (i.e., $ \bH^0$ equipped with the weak topology). 

Our goal now is to pass to the limit in \eqref{eq:URDGalerkinApprox} tested against some $\phi \in \bH^2$ as $N$ tends to infinity. Clearly, 
\begin{align*}
|\clA^{N,1}_{st} \phi - \clA^{1}_{st} \phi|_0  & \leq  \left| P_N  P \left[( \sigma_k \cdot \nabla) \phi\right]   Z_{st}^{N,k} -   P \left[(\sigma_k \cdot \nabla) \phi\right]   Z_{st}^{k}  \right|_{0}  + \left| P_N  P \left[( \nabla \sigma_k ) \phi\right]   Z_{st}^{N,k} -   P \left[( \nabla \sigma_k ) \phi\right]   Z_{st}^{k}  \right|_{0} \\
& \leq | P_N  P \left[(\sigma_k \cdot \nabla) \phi\right]  ( Z_{st}^{N,k} -    Z_{st}^{k} ) |_{0}  +  | (I - P_N)  P \left[(\sigma_k \cdot \nabla) \phi\right]   Z_{st}^{k} |_{0} \\
& + | P_N  P \left[(\nabla \sigma_k) \phi\right]  ( Z_{st}^{N,k} -    Z_{st}^{k} ) |_{0}  +  | (I - P_N)  P \left[(\nabla \sigma_k) \phi\right]   Z_{st}^{k} |_{0}.
\end{align*}
Making use of \eqref{A1bound}, we get
\begin{align*}
| P_NP(\sigma_k \cdot \nabla) \phi  |_{0} | Z_{st}^{N,k} -    Z_{st}^{k} | \lesssim_{N_0} | \phi  |_{1} | Z_{st}^{N} -    Z_{st}|, \quad | P_NP( \nabla \sigma_k ) \phi  |_{0} | Z_{st}^{N,k} -    Z_{st}^{k} | \lesssim_{N_0} | \phi  |_{0} | Z_{st}^{N} -    Z_{st}| 
\end{align*}
which both converge to $0$ as $N \rightarrow \infty$. For the remaining terms we notice that $P_N$ converges to $I$ with respect to  the strong topology on $\clL(\bH^0,\bH^0)$. Now, since $\phi \in \bH^2$, this implies that 
$$
\clA^{N,1}_{st} \phi \rightarrow \clA^{1}_{st} \phi \qquad \textrm{in } \quad \bH^0  \quad \textrm{ as } N \rightarrow \infty.
$$
In a similar way one can show that 
$$
\clA^{N,2}_{st} \phi \rightarrow \clA^{2}_{st} \phi \qquad \textrm{in } \quad \bH^0  \quad \textrm{ as } N \rightarrow \infty.
$$
and hence  
\begin{align*}
|(u^N_s,\clA_{st}^{N,i,*}\phi) - (u_s,\clA_{st}^{P,i,*}\phi)|&\le_{N_0}  | (u^N_s-u_s,\clA_{st}^{N,i,*}\phi) - (u_s,(\clA_{st}^{P,i,*}- \clA_{st}^{N,i,*})\phi)|\\
&\lesssim_{N_0}  |u^N_s-u_s|_{0}|\phi|_2+|u_s|_{0}|(\clA_{st}^{P,i,*}- \clA_{st}^{N,i,*})\phi|_{0}\rightarrow 0
\end{align*}
as $N \rightarrow \infty$.
Finally, using the  strong convergence in $L_{T^{*}}^2\bH^1$ of $\{u^N\}$ and \eqref{trilinear form estimate},  we find
\begin{align*}
\left|\int_s^t \left[B_P(u_r)(\phi) - B_P(u_r^N) (\phi) \right]\,dr \right|    &\leq \left| \int_s^t B_P(u_r - u_r^N,u_r)(\phi)\,dr \right| 
\quad + \left|\int_s^t  B_P(u_r^N,u_r- u_r^N) (\phi) \, dr \right| \\
&\lesssim  \int_s^t |u_r - u_r^N|_1 |u_r|_1 \, dr |\phi|_1    +  \int_s^t |u_r - u_r^N|_1 |u^N_r|_1 \, dr |\phi|_1 \rightarrow 0
\end{align*}
as $N \rightarrow \infty$.

Since all of the terms in equation \eqref{eq:URDGalerkinApprox} converge when applied to $\phi$,  the remainder $u^{N, \natural}_{st}(\phi)$ converges to some limit $u^{\natural}_{st}(\phi)$. 
Owing to the uniform bound \eqref{UniformGalerkinURD} in connection with \eqref{eq:GalerkinRemainder} we deduce that the limit $ u^{\natural} \in C^{\frac{p}{3}-\textnormal{var}}_{2, \varpi,L}([0,{T^{*}}]; \bH^{-2})$ for some control $\varpi$ depending only on $\omega_Z$ and $L> 0$ depending only on $p$, which proves that $u$ is a strong solution to \eqref{eq:RNSDiffForm}.
\end{proof}

\appendix

\section{Rough Gronwall lemma}
In this section, we formulate two Gronwall inequality involving controls.  The first one is a slight generalization of the Gronwall inequalities proved in \cite{DeGuHoTi16} and \cite{HH17}, and can be proved by the same reasoning. The second inequality is a corollary of the first inequality and the classical Bihari-LaSalle inequality.

\begin{lemma} \label{lem:RoughGronwall}
Assume that $G: [0,T] \rightarrow \bR_+$ is such that there exists constants $L>0$ and $\kappa>0$, and a  control $\omega$ such that for every $(s,t)\in \Delta_T$ with $\omega(s,t) \leq L$,
\begin{equation}
\delta G_{st} \leq \omega(s,t)^{\frac{1}{\kappa}} \sup_{0\le r\le t}G_t + \phi(s,t), 
\end{equation}
where  $\phi : \Delta_T \rightarrow \bR_+$ is such that $\phi(s,t) \leq \phi(0,T)$. Then there exists a constant $K>0$ depending only on $\omega$ such that 
$$
\sup_{ 0 \leq t \leq T} G_{t} \leq 2 \exp \left\{ \frac{\omega(0,T)}{L\alpha} \right\} \left( G_0 + K \phi(0,T)\right)
$$
where $\alpha :=  1 \vee L^{-1}(2 e^2)^{-\kappa}$.
\end{lemma}

\begin{remark}
The only difference between this version and the version in \cite{DeGuHoTi16} is that we do not require $\phi$ to be a control, or even superadditive as in \cite{HH17}.
\end{remark}

\begin{lemma}\label{lem:RoughBihari}
Assume that $W:\bR_+\rightarrow \bR_+$ is a non-decreasing continuous function with $W>0$ on $(0,\infty)$. Moreover, assume that  $G: [0,T] \rightarrow \bR_+$ is such that there exists constants $L>0$ and $\kappa>0$, and a  control $\omega$ such that for every $(s,t)\in \Delta_T$ with  $\omega(s,t) \leq L$, we have 
$$
\delta G_{s t} \leq C \int_s^t W\left(\sup_{s\le r' \le r}G_{r'}\right) dr  + \omega(s,t)^{\frac{1}{\kappa}} \sup_{0\le r\le t}G_{r}.
$$
Then
$$
\sup_{t\le T^*}G_{t} \leq W^{-1} \left( W\left(q  G_0\right) + T^*C q\right)
$$
where 
$$
q:=2 \exp \left\{ \frac{\omega(0,T)}{L\alpha}\right\},\;\;\alpha :=  1 \vee L^{-1}(2 e^2)^{-\kappa},
$$
$W$ is chosen such that $W'(x) = (w(x))^{-1}$, and $T^*>0$ is such that 
$$
W(q G_0) + T^*C q \in Dom( W^{-1}) . 
$$
\end{lemma}

\begin{proof}
Letting $\phi(s,t) = C \int_s^t W(\sup_{s\le r'\le r}G_{r'}) dr $,  we get from Lemma \ref{lem:RoughGronwall} that 
$$
G_{ \leq t} \leq q G_0 + q C \int_s^t W\left(\sup_{s\le r'\le r}G_{r'}\right) dr  \leq q G_0 + q C  \int_s^t W\left(\sup_{s\le r'\le r}G_{r'}\right) dr  .
$$
The result now follows from the classical Bihari-LaSalle inequality.
\end{proof}

\section*{Acknowledgement}

We are  enormously grateful for  helpful and inspiring discussions with Dan Crisan, Darryl Holm, Peter Friz, and Remigijus Mikulevicius.

\bibliographystyle{alpha}
\bibliography{bibliography}

\newcommand{\etalchar}[1]{$^{#1}$}
\begin{thebibliography}{MTVE03}

\bibitem[AMR12]{abraham2012manifolds}
Ralph Abraham, Jerrold~E Marsden, and Tudor Ratiu.
\newblock {\em Manifolds, tensor analysis, and applications}, volume~75.
\newblock Springer Science \& Business Media, 2012.

\bibitem[BCF91]{brzezniak1991stochastic}
Z.~Brze{\'z}niak, M.~Capi{\'n}ski, and F.~Flandoli.
\newblock Stochastic partial differential equations and turbulence.
\newblock {\em Mathematical Models and Methods in Applied Sciences},
  1(01):41--59, 1991.

\bibitem[BCF92]{brzezniak1992stochastic}
Z.~Brze{\'z}niak, M.~Capi{\'n}ski, and F.~Flandoli.
\newblock Stochastic {N}avier-{S}tokes equations with multiplicative noise.
\newblock {\em Stochastic Analysis and Applications}, 10(5):523--532, 1992.

\bibitem[BG17]{BaGu15}
Isma{\"e}l Bailleul and Massimiliano Gubinelli.
\newblock {Unbounded rough drivers}.
\newblock {\em {Annales de la Facult{\'e} des Sciences de Toulouse.
  Math{\'e}matiques.}}, 26(4):795--830, 2017.

\bibitem[BRS17]{BAILLEUL20175792}
I.~Bailleul, S.~Riedel, and M.~Scheutzow.
\newblock Random dynamical systems, rough paths and rough flows.
\newblock {\em Journal of Differential Equations}, 262(12):5792 -- 5823, 2017.

\bibitem[CCH{\etalchar{+}}19]{cotter2019numerically}
Colin Cotter, Dan Crisan, Darryl~D Holm, Wei Pan, and Igor Shevchenko.
\newblock Numerically modeling stochastic lie transport in fluid dynamics.
\newblock {\em Multiscale Modeling \& Simulation}, 17(1):192--232, 2019.

\bibitem[CFH17]{crisan2017solution}
Dan Crisan, Franco Flandoli, and Darryl~D Holm.
\newblock Solution properties of a 3d stochastic euler fluid equation.
\newblock {\em Journal of Nonlinear Science}, pages 1--58, 2017.

\bibitem[CGH17]{cotter2017stochastic}
Colin~J Cotter, Georg~A Gottwald, and Darryl~D Holm.
\newblock Stochastic partial differential fluid equations as a diffusive limit
  of deterministic {L}agrangian multi-time dynamics.
\newblock {\em Proc. R. Soc. A}, 473(2205):20170388, 2017.

\bibitem[Dav10]{Davie}
A.M. Davie.
\newblock Differential equations driven by rough paths: an approach via
  discrete approximation.
\newblock {\em Appl. Math. Res. eXpress 35,}, 2010.

\bibitem[DGHT16]{DeGuHoTi16}
A.~Deya, M.~Gubinelli, M.~Hofmanov{\'a}, and S.~Tindel.
\newblock A priori estimates for rough {PDE}s with application to rough
  conservation laws.
\newblock {\em arXiv preprint arXiv:1604.00437}, 2016.

\bibitem[DL89]{DiPernaLions}
R.J. DiPerna and P.L. Lions.
\newblock Ordinary differential equations, transport theory and sobolev spaces.
\newblock {\em Invent. math. 98, 511-547}, 1989.

\bibitem[FG95]{flandoli1995martingale}
F.~Flandoli and D.~Gatarek.
\newblock Martingale and stationary solutions for stochastic {N}avier-{S}tokes
  equations.
\newblock {\em Probability Theory and Related Fields}, 102(3):367--391, 1995.

\bibitem[FH14]{FrHa14}
P.~K. Friz and M.~Hairer.
\newblock {\em A course on rough paths: with an introduction to regularity
  structures}.
\newblock Universitext. Springer, Cham, 2014.

\bibitem[FPD{\etalchar{+}}14]{faranda2014modelling}
Davide Faranda, Flavio Maria~Emanuele Pons, B{\'e}rengere Dubrulle,
  Fran{\c{c}}ois Daviaud, Brice Saint-Michel, {\'E}ric Herbert, and
  Pierre-Philippe Cortet.
\newblock Modelling and analysis of turbulent datasets using auto regressive
  moving average processes.
\newblock {\em Physics of Fluids}, 26(10):105101, 2014.

\bibitem[FV10]{FrVi10}
P.~K. Friz and N.~B. Victoir.
\newblock {\em Multidimensional stochastic processes as rough paths: theory and
  applications}, volume 120.
\newblock Cambridge University Press, 2010.

\bibitem[HH18]{HH17}
Antoine Hocquet and Martina Hofmanov{\'a}.
\newblock An energy method for rough partial differential equations.
\newblock {\em Journal of Differential Equations}, 265(4):1407--1466, 2018.

\bibitem[HLN17]{HLN}
M.~Hofmanov{\'a}, J-M. Leahy, and T.~Nilssen.
\newblock On the {N}avier-{S}tokes equation perturbed by rough transport noise.
\newblock {\em To appear in Journal of Evolution Equations. arXiv preprint
  arXiv:1710.08093}, 2017.

\bibitem[HN18]{HN}
A.~Hocquet and T.~Nilssen.
\newblock An {I}t\^o formula for rough partial differential equations.
  application to the maximum principle.
\newblock {\em arXiv preprint arXiv:1806.10427}, 2018.

\bibitem[HNS18]{HNS}
A.~Hocquet, T.~Nilssen, and W.~Stannat.
\newblock Generalized {B}urgers equation with rough transport noise.
\newblock {\em arXiv preprint arXiv:1804.01335}, 2018.

\bibitem[LCL07]{MR2314753}
T.~J. Lyons, M.~Caruana, and T.~L{\'e}vy.
\newblock {\em Differential equations driven by rough paths}, volume 1908 of
  {\em Lecture Notes in Mathematics}.
\newblock Springer, Berlin, 2007.
\newblock Lectures from the 34th Summer School on Probability Theory held in
  Saint-Flour, July 6--24, 2004, With an introduction concerning the Summer
  School by Jean Picard.

\bibitem[LSEO16]{lilly2016fractional}
Jonathan~M Lilly, Adam~M Sykulski, Jeffrey~J Early, and Sofia~C Olhede.
\newblock Fractional brownian motion, the mat{\'e}rn process, and stochastic
  modeling of turbulent dispersion.
\newblock {\em arXiv preprint arXiv:1605.01684}, 2016.

\bibitem[Mik02]{mikulevicius2003cauchy}
R~Mikulevicius.
\newblock On the {C}auchy problem for the stochastic {S}tokes equations.
\newblock {\em SIAM Journal on Mathematical Analysis}, 34(1):121--141, 2002.

\bibitem[MR04]{mikulevicius2004stochastic}
R.~Mikulevicius and B.~L. Rozovskii.
\newblock Stochastic {N}avier--{S}tokes equations for turbulent flows.
\newblock {\em SIAM Journal on Mathematical Analysis}, 35(5):1250--1310, 2004.

\bibitem[MR05]{mikulevicius2005global}
R.~Mikulevicius and B.~L. Rozovskii.
\newblock Global {L}2-solutions of stochastic {N}avier--{S}tokes equations.
\newblock {\em The Annals of Probability}, 33(1):137--176, 2005.

\bibitem[MTVE03]{majda2003systematic}
Andrew~J Majda, Ilya Timofeyev, and Eric Vanden-Eijnden.
\newblock Systematic strategies for stochastic mode reduction in climate.
\newblock {\em Journal of the Atmospheric Sciences}, 60(14):1705--1722, 2003.

\bibitem[RL15]{Rockner}
M.~R\"{o}ckner and W.~Liu.
\newblock {\em Stochastic Partial Differential Equations: An Introduction}.
\newblock Universitext. Springer International Publishing, 2015.

\bibitem[Tao16]{tao2016finite}
Terence Tao.
\newblock Finite time blowup for lagrangian modifications of the
  three-dimensional euler equation.
\newblock {\em Annals of PDE}, 2(2):9, 2016.

\bibitem[Tay13]{taylor2013partial}
Michael Taylor.
\newblock {\em Partial differential equations II: Qualitative studies of linear
  equations}, volume 116.
\newblock Springer Science \& Business Media, 2013.

\bibitem[Tem83]{RT83}
R.~Temam.
\newblock {\em {N}avier-{S}tokes equations and nonlinear functional analysis},
  volume~41 of {\em CBMS-NSF Regional Conference Series in Applied
  Mathematics}.
\newblock Society for Industrial and Applied Mathematics (SIAM), Philadelphia,
  PA, 1983.

\bibitem[ZB16]{BFM}
Mario~Maurelli Zdzislaw~Brze{\'z}niak, Franco~Flandoli.
\newblock Existence and uniqueness for stochastic 2d {E}uler flows with bounded
  vorticity.
\newblock {\em Archive for Rational Mechanics and Analysis}, 221:107 -- 142,
  2016.

\bibitem[ZZ15]{ZHU20154443}
Rongchan Zhu and Xiangchan Zhu.
\newblock Three-dimensional {N}avier-{S}tokes equations driven by space-time
  white noise.
\newblock {\em Journal of Differential Equations}, 259(9):4443 -- 4508, 2015.

\end{thebibliography}
\end{document}